\documentclass[10pt]{article}

\RequirePackage{amsthm,amsmath,amsfonts,amssymb}
\RequirePackage[numbers]{natbib}
\RequirePackage[colorlinks,citecolor=blue,urlcolor=blue]{hyperref}
\RequirePackage{graphicx}

\usepackage{bbm}
\usepackage[margin=2.54cm]{geometry}
\usepackage[utf8]{inputenc}
\usepackage[T1]{fontenc}
\usepackage{amsmath} 
\usepackage{amssymb}
\usepackage{paralist}
\usepackage{hyperref}
\usepackage{amsthm}
\usepackage{color}
\usepackage{mathtools}
\usepackage{accents}
\usepackage{dsfont}

\theoremstyle{plain}

\newtheorem{theorem}{Theorem}[section]
\newtheorem{lemma}[theorem]{Lemma}
\newtheorem{corollary}[theorem]{Corollary}
\newtheorem{proposition}[theorem]{Proposition}
\theoremstyle{remark}
\newtheorem{definition}[theorem]{Definition}
\newtheorem{remark}[theorem]{Remark}

\newcommand{\discmu}{{\bar{\mu}}}
\newcommand{\discsigma}{\mathcal{F}}

\newcommand{\overcirc}{\accentset{\circ}}

\newcommand{\extmu}{{\tilde{\mu}}}

\newcommand{\gradsigma}{\mathcal{E}}

\renewcommand{\d}{\mathrm{d}}

\renewcommand{\k}{\kappa}

\renewcommand{\t}{\mathbf{t}}

\newcommand{\R}{{\mathbb{R}}}
\renewcommand{\P}{{\mathbb{P}}}
\newcommand{\E}{{\mathbb{E}}}

\newcommand{\Z}{{\mathbb{Z}}}

\newcommand{\bs}{\boldsymbol}

\newcommand{\p}{\varphi}
\newcommand{\crm}{\mathrm{c}}

\newcommand{\Edge}{\mathbf{E}}

\newcommand{\Vertex}{\mathbf{V}}
\newcommand{\1}{\mathds{1}}
\newcommand{\vertiii}[1]{{\left\vert\kern-0.25ex\left\vert\kern-0.25ex\left\vert
#1 
    \right\vert\kern-0.25ex\right\vert\kern-0.25ex\right\vert}}

\newcommand{\Acal}{\mathcal{A}}
\newcommand{\Bcal}{\mathcal{B}}

\newcommand{\Dcal}{\mathcal{D}}

\newcommand{\Mcal}{\mathcal{M}}

\begin{document}

\title{Aizenman-Wehr argument 
for a class of disordered gradient models}
\author{Simon Buchholz 
\footnote{Max Planck Institute for Intelligent Systems,
T\"ubingen, sbuchholz@tue.mpg.de}
 \and Codina Cotar\footnote{Department of Statistical Science, University College London, c.cotar@ucl.ac.uk}
}
\date{\today}
\maketitle
\begin{abstract}
We consider random gradient fields with disorder where the interaction potential $V_e$ on an edge $e$ 
can be expressed as
\begin{align*}
e^{-V_e(s)} = \int \rho(\d\kappa)\, e^{-\kappa \xi_e} e^{-\frac{\kappa s^2}{2}}.
\end{align*}
Here $\rho$ denotes a measure with compact support in $(0,\infty)$ and
$\xi_e\in\R$ a nontrivial edge dependent disorder. We show that in dimension $d=2$
there is a unique shift covariant disordered gradient Gibbs measure such that the annealed measure is ergodic and has zero tilt. 
This shows that the phase transitions known to occur for this class of potential do not persist to the  disordered setting. 
The proof relies on the connection of the gradient Gibbs measures to  a random conductance model with compact state space, to which the well known Aizenman-Wehr argument applies.
\end{abstract}

\section{Introduction}\label{sec:introduction}
Gradient fields are statistical mechanics models that can be used to model phase separation 
or, in the case of vector valued fields, solid materials. 
Formally they can be defined as a random field $(\p_x)_{x\in \Z^d}\in \R^{\Z^d}$ with the Gibbs distribution
\begin{align}\label{eq:formal_grad_int_def}
\frac{\exp\left(-\sum_{x\sim y} V(\p(x)-\p(y))\right)}{Z}\prod_{x\in\Z^d} \d\p(x).
\end{align}
Here $\d\p(x)$ denotes the Lebesgue measure, $V:\R\to\R$ is a measurable symmetric potential, and $\sim$ indicates the neighborhood relation for $\Z^d$.
We can give a meaning to the formal expression \eqref{eq:formal_grad_int_def}
using the DLR-formalism. 
In the setting of gradient interface models no Gibbs measure exists in dimension $d\leq 2$.
Therefore one often considers gradient Gibbs measures which exist in all dimensions.
This means that attention is restricted to the $\sigma$-algebra generated by the gradient fields
\begin{align}
\eta_{xy}=\p(y)-\p(x) \quad \text{for $x\sim y$.}
\end{align} 
  
In the case where the potential $V$ satisfies $c_1\leq V''(s)\leq c_2$ for some $0<c_1<c_2$ and all $x\in \R$, in particular $V$ is uniformly convex, gradient fields are quite well understood. One important result shown in \cite{MR1463032} is that the tilt
is in bijection with the ergodic gradient Gibbs measures where the tilt $u\in \R^d$ of a shift invariant gradient measure $\mu$ is defined as
\begin{align}
\E_\mu(\nabla \p(x))=u
\end{align}
where $\nabla\p(x)\in\R^d$ denotes the discrete derivative, i.e., the vector with entries $\nabla_i\p(x)=\p(x+e_i)-\p(x)$.
For non-convex potentials much less is known. Exceptions are results that show that for very high temperatures the model behaves essentially as in the convex case
\cite{Cotar2009, Cotar2012, MR3913274} and partial results in this direction for very small temperatures \cite{adams2016strict, adams2019strict}. 
 For intermediate temperatures no general results are known and all results
 are restricted to the class of potentials $V:\R\to\R$ that admit the representation
 \begin{align}\label{eq:defofpot}
e^{-V(s)} = \int \rho(\d\kappa)\, e^{-\frac{\kappa s^2}{2}}
 \end{align}
for some measure $\rho$ with support in $(0,\infty)$.
In particular, in \cite{MR2778801}, where this class was introduced 
it was shown that for zero-tilt and suitable $\rho$ there exist two ergodic gradient Gibbs measures. This establishes that new phenomena arise for non-convex potentials.
The main goal of this note is to study the same class of potentials with an additional disorder.

Disordered gradient interface models were introduced in \cite{MR2985173}
where existence of disordered gradient Gibbs measure was shown in a rather general setting. Moreover, these results were extended in \cite{MR3383338} to show uniqueness of the disordered gradient Gibbs measures for strictly convex potentials.
Here we extend the uniqueness results to certain disordered non-convex potentials.

Disordered models have been studied in many contexts in statistical mechanics, e.g., in the context of the random field Ising model or spin glasses (see \cite{MR2252929} for some general background). 
Imry and Ma predicted in \cite{Imry:1975zz} that phase transitions do not occur in disordered models in dimension 2. Their argument is based on the heuristic that the effect of a pointwise  disorder on a volume $\Lambda$ is governed by the central limit theorem hence of order $\sqrt{|\Lambda|}$ while
the boundary effect is of order $|\partial \Lambda|$. In dimension 2 they both scale in the same way and they argued that the disorder dominates. 
Their reasoning was made rigorous by Aizenman and Wehr in \cite{MR1060388} where they show for a large class of models (in particular the random field Ising model) that there is a unique Gibbs state for almost all realizations of the disorder. In general, their argument does not apply to models with unbounded state space, e.g., gradient models. 
Moreover there are several technical difficulties in the absence of correlation inequalities which do not hold for gradient models with non-convex interactions.
Here we nevertheless extend the result to certain gradient models. The main ingredient in our proof is that for potentials as in \eqref{eq:defofpot} the
aforementioned difficulties can be avoided by coupling the gradient model to a random conductance model with compact state space. 
This was already remarked in \cite{MR2778801} and further investigated 
in \cite{phasetransitionclass} where strong correlation inequalities for the model were shown. In this work we extend those results to the disordered setting and then show that the Aizenman-Wehr result can be applied to the disordered random conductance model   with small modifications. This result can be lifted to  conclude uniqueness also for the gradient Gibbs measures.

The rest of this article is structured as follows. In Section \ref{sec:gradGibbs}
we introduce disordered gradient Gibbs measures and state our main result, the uniqueness of disordered gradient Gibbs measures for potentials as in
\eqref{eq:defofpot} in dimension $d=2$. Then, in Section \ref{sec:RC} we introduce the random conductance model and prove correlation inequalities.
The next section contains the proof of the uniqueness result for the random 
conductance model.
Finally, in the slightly technical Section \ref{sec:GMRC}, we relate 
the gradient Gibbs measures to the random conductance model
and finish the proof of our main result.


\section{Gradient Gibbs measures with disorder}\label{sec:gradGibbs}
In this section we recall the notion of gradient Gibbs measure with disorder
that was introduced in \cite{MR2985173}. This generalizes the notion
of gradient Gibbs measures and we refer to  \cite{MR2807681, MR2251117} for general background on gradient Gibbs measures.
We will also state our main results for a special class of gradient models.
Our presentation follows \cite{MR3383338} and in fact our 
disordered gradient model can be stated in terms of their 
model B. However, their uniqueness results do not apply in our setting because
the involved potentials are not convex. 

We start with some general notation.
In this work we consider real valued random fields that are indexed by a subset $\Lambda \subset \Z^d$. We  denote the set of nearest neighbor bonds of $\Z^d$ by $\Edge(\Z^d)$. 
To consider gradient fields it is useful to choose on orientation of the edges.
We orient the edges $e=\{x,y\}\in \Edge(\Z^d)$
from $x$ to $y$ iff $x\leq y$ (coordinate-wise), i.e., we can view the graph $(\Z^d,\Edge(\Z^d))$ as a directed graph but
mostly we work with the undirected graph.

We associate to a random field $\p:\Z^d\to \R$  the 
gradient field $\eta=\nabla \p\in \R^{\Edge(\Z^d)}$ given by $\eta_{e}=\p_y-\p_x$ if $\{x,y\}\in \Edge(\Z^d)$
are nearest neighbors and $x\leq y$. We formally write $\eta_{x,y}=\eta_e=\p_y-\p_x$
and $\eta_{y,x}=-\eta_e=\p_x-\p_y$.
The gradient field $\eta$ 
 satisfies the plaquette condition
\begin{align}\label{eq:plaquette}
\eta_{x_1,x_2}+\eta_{x_2,x_3}+\eta_{x_3,x_4}+\eta_{x_4,x_1}=0
\end{align}
for every plaquette, i.e., nearest neighbors $x_1,x_2,x_3,x_4,x_1$.
Vice versa, given a field $\eta\in\R^{\Edge(\Z^d)}$ that satisfies the plaquette condition 
there is a up to constant shifts a unique field $\p$ such that $\eta=\nabla \p$ (the antisymmetry
of the gradient field is contained in our definition).
We will refer to those fields as gradient fields and denote them by
$\R^{\Edge(\Z^d)}_g$.
It is sometimes convenient to identify $\R^{\Edge(\Z^d)}_g$ with the
space of fields $\R^{\Z^d}$ such that $\p(0)=0$.

To simplify the notation we write $\p_\Lambda$ for $\Lambda\subset \Z^d$ and $\eta_{E}$
for $E\subset \Edge(\Z^d)$
for the the restriction of fields and gradient fields.

Later it will be convenient to work with general connected graphs $G=(V,E)$ and 
we will use the notation $V=\Vertex(G)$ and $E=\Edge(G)$ for the vertices and edges of the graph. 
 We identify a subset   $\Lambda\subset \Z^d$ with the graph generated by it and
therefore we write  $\Edge(\Lambda)$ for the bonds with both endpoints in
$\Lambda$.

For a subset $H\subset \Z^d$ we write $\partial H$ for the (inner) boundary of $H$ consisting of all points
$x\in \Vertex(H)$ such that there is an edge $e=\{x,y\}\in \Edge(\Z^d)\setminus \Edge(H)$.
In the case of a graph generated by $\Lambda\subset \Edge(\Z^d)$ we have $x\in \partial\Lambda$ if
there is $y\in \Lambda^{\crm}$ such that $\{x,y\}\in \Edge(\Z^d)$.
We define $\accentset{\circ} \Lambda=\Lambda\setminus \partial\Lambda$. 
For a finite subset $\Lambda\subset \Z^d$ we denote by $\d\p_\Lambda$ the Lebesgue measure on $\R^\Lambda$. 
We now define finite volume disordered Gibbs measures.
We consider a probability space $(\Omega, \Bcal, \P)$ for the disorder.
In the following expectations $\E$ will always be with respect to $\P$ 
and allows us to pass from quenched to annealed quantities. Other expectations that appear
will be denoted by integrals or by the action of measures.
 We assume that there is for each $e\in \Edge(\Z^d)$ a measurable map $V^\xi_e (s):\Omega\times \R\to \R$.
We assume that $V^\xi_e$ is even and satisfies $V^\xi_e(s)\geq As^2 + B$ for all $\xi\in \Omega$, $e\in \Edge(\Z^d)$, and $s\in \R$ for some constants $A>0$ and $B\in \R$. 
The relevant energy
is given by the disordered Hamiltonian 
\begin{align}
H_\Lambda[\xi](\p)=  \sum_{e\in \Edge(\Lambda)}  V^\xi_e((\nabla\p)_e).
\end{align}
\begin{definition}[Finite volume $\p$-Gibbs measure]
For $\psi:\Z^d\to \R$ we define the finite volume Gibbs measure $\nu_\Lambda^\psi[\xi](\d \p)$ on 
$\R^{\Z^d}$ by
\begin{align}
\nu_\Lambda^\psi[\xi](\d\p)
=
\frac{1}{Z_\Lambda^\psi[\xi]} e^{-H_\Lambda[\xi](\p)} \prod_{x\in \overcirc{\Lambda}^\crm} \delta_{\psi(x)}(\d \p(x)) \;\d\p_{\overcirc{\Lambda}}.
\end{align}
Here $Z_\Lambda^\psi[\xi]$ denotes the partition function that turns the measure into a probability measure. Those expressions are well-defined under the assumptions stated above. 
\end{definition}
We define infinite volume Gibbs measures using the DLR-condition. 
\begin{definition}[$\p$-Gibbs measure on $\Z^d$]
A probability measure $\nu[\xi]$ on $\R^{\Z^d}$ is a Gibbs measure for the disorder $\xi$ if it satisfies the DLR equation
\begin{align}
\int \nu[\xi](\d \psi) \int \nu^\psi_\Lambda[\xi](\d \p)\, F(\p)
=\int \nu[\xi](\d\p) \, F(\p)
\end{align}
for every finite $\Lambda\subset \Z^d$ and for all $F\in C_\mathrm{b}(\R^{\Z^d})$.
\end{definition}
It is well known that shift invariant Gibbs measures do not exist in dimension $d\leq 2$. Therefore one often restricts attention to gradient Gibbs measures which are measures on $\R^{\Edge(\Z^d)}_g$. This is also necessary if one tries to model tilted surfaces with $u\neq 0$.
We define gradient Gibbs measures in finite volume as follows.
\begin{definition}[Finite volume $\nabla\p$-Gibbs measure]
The finite volume gradient Gibbs measure $\mu_\Lambda^\sigma[\xi]$ 
for $\sigma \in \R_g^{\Edge(\Z^d)}$ and given disorder $\xi\in \Omega$ is the unique probability measure on $\R_g^{\Edge(\Z^d)}$ that satisfies for all $F\in C_b(\R_g^{\Edge(\Z^d)})$ 
\begin{align}
\label{finitevolgraddef}
\int \mu_\Lambda^\sigma[\xi](\d\eta)F(\eta)=\int \nu_\Lambda^\psi[\xi](\d\p)\,F(\nabla\p)
\end{align} 
where $\psi:\Z^d\to\R$ satisfies $\nabla \psi=\sigma$.
\end{definition}
It is easy to see that the definition above is independent of the choice of $\psi$. Indeed, $\psi$ is a unique up to a global shift that does not influence the gradient distribution $\nabla \p$ under $\nu_\Lambda^\psi[\xi]$.
The extension to infinite volume is now similar to the extension for usual Gibbs-measures.
\begin{definition}[$\nabla\p$-Gibbs measure on $\Edge(\Z^d)$]
A probability measure $\mu[\xi]$ on $\R^{\Edge(\Z^d)}_g$ is a gradient Gibbs measure for the disorder $\xi\in\Omega$ if it satisfies the DLR equation
\begin{align}
\int \mu[\xi](\d \sigma) \int \mu^\sigma_\Lambda[\xi](\d \eta)\, F(\eta)
=\int \mu[\xi](\d\eta) \, F(\eta)
\end{align}
for every finite $\Lambda\subset \Z^d$ and for all $F\in C_\mathrm{b}(\R^{\Edge(\Z^d)}_g)$.
\end{definition}
We now introduce the concept of shift covariance which is the natural analogue of shift invariance in the disordered setting. First we consider the case without disorder.
For $a\in \Z^d$ we define the shifts $\tau_a:\Z^d\to \Z^d$ and $\tau_a:\Edge(\Z^d)\to \Edge(\Z^d)$ (we always use the same symbol $\tau_a$ for shifts)  by
\begin{align}\label{eq:shift_Zd}
\tau_a x=x+a, \quad \tau_a (x,y)=(x+a,y+a).
\end{align}
Moreover we consider the extension $\tau_a:\R^{\Edge(\Z^d)}\to \R^{\Edge(\Z^d)}$  to gradient fields that is defined by 
\begin{align}\label{eq:shift_fields}
(\tau_a\eta)_{x,y}=\eta_{\tau_a^{-1}(x,y)}=\eta_{x-a,y-a}.
\end{align} 
  A measure $\mu$ is shift invariant if $\mu (\tau_a^{-1}(A))=\mu(A)$ for all $a\in\Z^d$ and $A\in \mathcal{B}(\R)^{\Edge(\Z^d)}$.
An event is shift invariant if $\tau_a(A)=A$ for all $a\in \Z^d$.
A shift invariant gradient measure is ergodic if $\mu(A)\in \{0,1\}$ for all shift invariant $A$.

We assume that there is an action also labelled by $\tau_a$ on $\Omega$ such that
\begin{align}\label{eq:shift_covariance_disorder}
V_{\tau_a e}^{\tau_a \xi}=V_e^\xi
\end{align}
 and that $\P$ is invariant under $\tau_a$.
\begin{definition}[Shift covariant gradient Gibbs measure]
A measurable map $\xi\to \mu[\xi]$ is a shift covariant gradient Gibbs measure if, for almost every $\xi$, $\mu[\xi]$ is a gradient Gibbs measure and 
\begin{align}
\int \mu[\tau_a \xi](\d\eta) \, F(\eta) =\int \mu[\xi](\d\eta) F(\tau_a\eta)
\end{align}
for all $a\in \Z^d$ and $F\in C_b(\R_g^{\Edge(\Z^d)})$.
\end{definition}
  We define the tilt $u\in \R^d$ of a shift covariant gradient Gibbs measure as the tilt of the annealed measure, i.e.,
  \begin{align}
  u_i = \mathbb{E}\left(\int \mu[\xi](\d\eta) \eta_{0,e_i}\right).
  \end{align}

  We now specialize to the setting we will consider in the following.  
  For simplicity we assume that $\Omega = \R^{\Edge(\Z^d)}$ and we assume that
  if $\xi\in \Omega$ is distributed according to $\P$ the coordinates $\xi_e$ and $\xi_f$ for $e\neq f \in \Edge(\Z^d)$ are i.i.d.\ random variables, in particular $\P$ is a product measure and $\P$ is shift invariant.   
  The action of the shift on $\Omega$ is defined similarly to the definition for gradient fields by 
  \begin{align}\label{eq:shift_disorder}
  (\tau_a\xi)_e=\xi_{\tau_a^{-1}e}.
  \end{align} 
We now fix a Borel measure $\rho$ on $\R_+$ with unit mass and we assume that there is a constant $\lambda>0$ such that $\mathop{supp} \rho \subset [\lambda^{-1},\lambda]$.  The disordered potential $V^\xi_e$ is defined by
\begin{align}\label{eq:defofpotdisordered}
e^{-V_e^\xi(s)} = \int_{\R_+} \rho(\d\kappa)\, e^{-\frac{\kappa s^2}{2}}e^{\xi_e\kappa}.
\end{align}
Using \eqref{eq:shift_Zd} and \eqref{eq:shift_disorder} this implies
$V^{\tau_a\xi}_{\tau_a e}=V_{a+e}^{ \xi_{\cdot-a}}=V_e^\xi$ and therefore the disorder satisfies \eqref{eq:shift_covariance_disorder}.
We can now state our main theorem.
\begin{theorem}\label{th:main}
Let $d=2$.
Assume that $\E(e^{t|\xi_e|})<\infty$ for all $t\in \R$ and
assume that the distribution of $\xi_e$ is not concentrated on a single point. Then
there exists for $V^\xi$ as above a unique shift covariant gradient Gibbs measure $\xi\to \mu[\xi]$ 
such that the annealed measure $\E(\mu[\xi])$ is ergodic and has zero tilt and satisfies 
\begin{align}\label{eq:cond_moment}
\E\left(\mu[\xi](|\eta_e|^{2+\varepsilon})\right)<\infty
\end{align}
for some $\varepsilon>0$ and all $e\in \Edge(\Z^2)$.
\end{theorem}

One important special case is for $\rho$ is the simplest non-trivial class of measures
\begin{align}\label{eq:defrhospecial}
\rho = p\delta_q+ (1-p)\delta_1
\end{align}
where $p\in [0,1]$ and $q\ge 1$. In this case it was shown in \cite{MR2778801} that in dimension $d=2$ there are two zero-tilt Gibbs states for $q$ sufficiently large and $p_{\mathrm{sd}}$ characterized as the solution of the equation $p^4/(1-p)^4 = q$.
As a special case of our main result we show that the coexistence disappears as soon as disorder is introduced.
\begin{corollary}
Fix $q> 1$ and $d=2$. Assume that $(p_e)_{e\in \Edge(\Z^2)}$ are i.i.d.\ random variables and there exists $\epsilon>0$ such that $p_e$ is supported in $[\epsilon,1-\epsilon]$.
Moreover, we assume that $p_e$ is not concentrated on a single point. Then, for $d=2$ there is a unique shift covariant zero-tilt gradient Gibbs measure $(p_e)_{e\in \Edge(\Z^2)}\to\mu[(p_e)_{e\in \Edge(\Z^2)}]$ for the random potential $V^{p_e}_e$ given by
\begin{align}
e^{-V_e^{p_e}(s)} = p_e e^{\frac{-qs^2}{2}} + (1-p_e)e^{-\frac{s^2}{2}}
\end{align}
such that 
\begin{align}
\E\left(\mu[(p_e)_{e\in \Edge(\Z^2)}](|\eta_e|^{2+\varepsilon})\right)<\infty
\end{align}
for some $\varepsilon>0$.
\end{corollary}
\begin{proof}
We consider the general setting above and take $\rho=\tfrac12 \delta_1+\tfrac12 \delta_q$. Then
\begin{align}
e^{-V_e^\xi(s)}
= \frac12 e^{\xi_e}e^{-\frac{s^2}{2}} + \frac12 e^{\xi_e q}e^{-\frac{qs^2}{2}}
= \frac12 \left(e^{\xi_e}+e^{\xi_e q}\right)
\left(p(\xi_e) 
e^{-\frac{s^2}{2}}
+(1-p(\xi_e)
e^{-\frac{qs^2}{2}}\right)
\end{align}
where
\begin{align}\label{eq:p_xi}
p(\xi_e)=\frac{e^{\xi_e}}{ e^{\xi_e}+e^{\xi_e q}}.
\end{align}
Note that this is a bijection from $(-\infty,\infty)$ to $(0,1)$ for $q>1$.
Moreover, 
$V_e^\xi(s)=V_e^{p(\xi_e)}+c(\xi_e)$ where the constant term does not change the   distribution of the gradient fields.
We obtain  i.i.d.\  
random variables $\xi_e$ such that $p(\xi_e)\overset{\Dcal}{=} p_e$ using the inverse function of
\eqref{eq:p_xi}. The fact that the distribution of $p_e$ is supported in 
$[\epsilon,1-\epsilon]$ implies that also $\xi_e$ has bounded support.
Therefore we can apply Theorem~\ref{th:main} to conclude. 
\end{proof}
The proof of the theorem can be found in Section  \ref{sec:GMRC}. 
Let us remark that we did not try to optimize the integrability condition for the disorder neither in the theorem nor in the corollary. In fact 
the general results require only $2+\varepsilon$ moments of $\xi_e$ (see, e.g., \cite{MR2252929}). 
We also expect that instead of \eqref{eq:cond_moment}
 square integrability is sufficient  to prove uniqueness.

\section{The disordered random conductance model}\label{sec:RC}
As in earlier works we can connect the gradient interface model
for potentials as in \eqref{eq:defofpotdisordered} with certain random conductance models. 
Here we introduce the relevant random conductance model in slightly
larger generality than in \cite{phasetransitionclass} and add disorder to the model. We keep this section as short as possible and refer to 
 \cite{phasetransitionclass} for a more detailed account and motivation for this
 model.
  The detailed investigation of the relation between the random conductance model and the gradient models is deferred to Section~\ref{sec:GMRC}. To provide a bit of context we note that for potentials as in \eqref{eq:defofpotdisordered}
  the decomposition 
  \begin{align}
  e^{-V_e^\xi(s)} = \int_{\R_+} \rho(\d\kappa)\, e^{-\frac{\kappa s^2}{2}}e^{\xi_e\kappa}
  \end{align}
  allows to consider extended Gibbs measures (see Section~\ref{sec:GMRC})
  for the joint distribution of the gradient field and the conductances $\kappa$.
  The key observation is that conditioned on $\kappa$ the gradient field is an inhomogeneous Gaussian field with whose partition function can be expressed as a certain determinant. This allows us to infer the density of the 
  $\kappa$-marginal which is given by the expression in \eqref{eq:kappa-marginal}
  below which we will study in this section.
  
We denote the set of positive Borel measures with mass 1  on $\R_+$ by $\Mcal_1(\R_+)$. 
It is helpful to consider the random conductance model on arbitrary connected finite graphs
$G=(\Edge(G), \Vertex(G))$.
For a set of conductances $\kappa\in \R^{\Edge(G)}$ we define the (combinatorial) graph Laplacian
$\Delta_\kappa$ 
acting on $H_0=\{f:\Vertex(G)\to \R: \, \sum_{x\in \Vertex(G)} f(x)=0\}$
by 
\begin{align}
\Delta_\kappa f(x)=\sum_{y\sim x} \kappa_{\{x,y\}}\left(f(x)-f(y)\right).
\end{align}
This is a positive linear operator and therefore $\det \Delta_\kappa>0$. 
Suppose we are given an element $\bs{\rho}$ of $\Mcal(\R_+)^{\Edge(G)}$.
Later $\bs{\rho}$ will depend on the realization of the disorder but all results 
in this section hold in this more general setting.
We consider the probability measure $\discmu^{\bs{\rho}}$ on $\R_+^{\Edge(G)}$ given by 
\begin{align}\label{eq:kappa-marginal}
\discmu^{\bs{\rho}}(\d \kappa) = \frac{1}{Z} \frac{1}{\sqrt{\det \Delta_{\kappa}}}\, \prod_{e\in \Edge(G)} \bs{\rho}_e(\d\kappa_e). 
\end{align} 
Here $Z$ denotes the normalization so that $\discmu^{\bs{\rho}}$ is a probability measure. We need to impose certain conditions on $\bs{\rho}$
to ensure that $Z$ is finite but actually we will later only consider $\bs{\rho}$ 
with compact support in $(0,\infty)$ which is certainly sufficient. 
Let us remark that the bar will always indicate measures on conductances, i.e., on $\R_+^E$ for a suitable edge set $E$. The notation $\discmu$ shall indicate that these measures
are in close relation to the gradient Gibbs measures $\mu$ and we will later provide an explicit map from
gradient Gibbs measures to measures on conductances.
To simplify the notation in the following we introduce the function $s_G:\R_+^{\Edge(G)}\to \R_+$ as a shorthand for the density
\begin{align}\label{eq:def_s}
s_G(\kappa) = \frac{1}{\sqrt{\det \Delta_\kappa}}.
\end{align}
%
%
%
%

To state the results of this section we need some notation.
Let $E$ be a finite or countable infinite set. Let $S=\R_+^E$
and $\mathcal{F}$ the $\sigma$-algebra generated by cylinder events 
and $\mathcal{F}_{E'}$ the $\sigma$-algebra generated by the cylinder events in $\R_+^{E'}\hookrightarrow
\R_+^E$. 
We consider the usual partial order on $S$ given by $\kappa^1\leq \kappa^2$ iff $\kappa^1_e\leq \kappa^2_e$ for all $e\in E$. A function $X:S\to \R$ is increasing if
$X(\kappa_1)\leq X(\kappa_2)$  for $\kappa_1\leq \kappa_2$ and decreasing if $-X$ is increasing.
An event $A\subset S$ is increasing if its indicator function is increasing.
We write $\discmu_1\succsim  \discmu_2$ if $\discmu_1$ stochastically dominates $\discmu_2$ which is by Strassen's Theorem equivalent to the existence of a coupling $(\kappa_1,\kappa_2)$ such that $\kappa_1\sim \discmu_1$ and $\kappa_2\sim\bar \mu_2$ and $\kappa_1\geq \kappa_2$ (see \cite{MR177430}). We introduce the minimum $\kappa_1 \wedge \kappa_2$ and the maximum $\kappa_1\vee \kappa_2$ of two configurations given by
$(\kappa_1 \wedge \kappa_2)_e=\min((\kappa_1)_e, (\kappa_2)_e)$
and 
$(\kappa_1 \vee \kappa_2)_e=\max((\kappa_1)_e, (\kappa_2)_e)$ for any $e\in E$.
Finally we introduce for $f,g\in E$ and $\kappa\in S$ the notation $\kappa_{fg}^{\pm \pm}\in S$ for the configuration given by
$(\kappa_{fg}^{\pm\pm})_e=\kappa_e$ for $e\notin \{f,g\}$ 
and $(\kappa_{fg}^{\pm \ast})_f= c_f^{\pm}$, $(\kappa_{fg}^{\ast \pm})_g= c_g^{\pm}$ where $c_f^\pm, c_g^\pm\in \R_+$ are constants satisfying 
$c_f^-<c_f^+$  and $c_g^-<c_g^+$ that are suppressed in the notation.
We define $\kappa_f^{\pm}$ similarly. We sometimes drop the edges $f$, $g$ from the notation.
We write $\discmu(X)=\int_{\Omega}X\,\d\discmu$ for $X:\Omega\to \R$.
In the context of continuum models it is possible to state correlation inequalities for probability measures that can be expressed as $\discmu = s \bs\rho$ where $s:S\to \R_+$ is a density function and $\bs \rho$ is a product measure on $S=\R_+^E$. 

\begin{theorem}[Holley inequality]\label{th:stoch_dom}
Let $E$ be finite and $\discmu_1=s_1\bs\rho$, $\discmu_2=s_2\bs\rho$  two probability measures on $S$   that  satisfy the Holley inequality
\begin{align}\label{eq:HolleyTheorem}
s_2(\kappa_1\vee \kappa_2)s_1(\kappa_1\wedge \kappa_2)\geq s_1(\kappa_1)s_2(\kappa_2) \quad\text{for $\kappa_1,\kappa_2\in S$}.
\end{align}
Then $\discmu_1\precsim \discmu_2$.
\end{theorem}
\begin{proof}
The original proof for the discrete setting $S=\{0,1\}^E$ appeared in \cite{MR0341552}, for the extension to continuum models we refer to \cite[Theorem 3]{MR341553}
and also \cite[Theorem 2]{MR0467867} for a different proof.
\end{proof}
Let us also state the FKG inequality.

\begin{theorem}\label{th:FKG}
A probability measure $\discmu=s \bar\rho$ that satisfies the lattice condition
\begin{align}\label{eq:lattice_FKG}
s(\kappa_1\vee\kappa_2) s(\kappa_1\wedge \kappa_2)\geq s(\kappa_1)s(\kappa_2) \quad\text{ for $\kappa_1,\kappa_2\in S$}.
\end{align}
satisfies the FKG inequality, i.e., for increasing functions $X,Y:S\to \R$  such that the following integrals exist we have
\begin{align}\label{eq:FKG}
\discmu(XY)\geq \discmu(X)\discmu(Y).
\end{align}
\end{theorem}
\begin{proof}
A proof can be found in  \cite[Theorem 2.16]{MR2243761} for the case $S=\{0,1\}^E$
and we refer to \cite[Corollary 2]{MR0467867} for an extension to the continuum setting.
\end{proof}
The next theorem provides a simple way to verify the assumptions
of Theorem \ref{th:stoch_dom} and Theorem \ref{th:FKG}.
Basically it states that it is sufficient to check the conditions when varying at most two edges.
\begin{theorem}\label{th:stoch_dom_crit}
Let  $\discmu_1=s_1\bs\rho$, $\discmu_2=s_2 \bs\rho$ be probability measures on $S$.
Then $s_1>0$ and $s_2>0$ satisfy \eqref{eq:HolleyTheorem}
iff the following two inequalities hold
\begin{align}
s_2(\k_f^+)s_1(\k_f^-)&\geq 
s_1(\k_f^+)s_2(\k_f^-),\quad \text{for $\kappa\in S$, $f\in E$, $c_f^-<c_f^+$} \label{eq:stoch_dom_crit1}\\
s_2(\k_{fg}^{++})s_1(\k_{fg}^{--})
&\geq s_1(\k_{fg}^{+-})s_2(\k_{fg}^{-+}),\quad \text{for $\k\in S$, $f,g\in E$, $c_f^-<c_f^+$, $c_g^-<c_g^+$}. \label{eq:stoch_dom_crit2}.
\end{align}
In particular, \eqref{eq:stoch_dom_crit1} and \eqref{eq:stoch_dom_crit2}
together imply $\discmu_1\precsim\discmu_2$.
\end{theorem}
\begin{proof}
See  \cite[Theorem 2.3]{MR2243761}. The proof uses an elementary induction argument in the 	Hamming distance of $\k_1$ and $\k_2$ and directly extends to the continuum setting.
\end{proof}
We state one simple corollary of the previous results.
\begin{corollary}\label{co:stoch_dom}
Let $\discmu_1=s_1\bs\rho$, $\discmu_2=s_2\bs\rho$ be two probability measures on $S$ such that 
 $s_1>0$ and  $s_2>0$ and either $s_1$ or $s_2$  satisfies \eqref{eq:lattice_FKG}.
 Then
\begin{align}\label{eq:stoch_dom_crit1_repeated}
s_2(\k_f^+)s_1(\k_f^-)&\geq 
s_1(\k_f^+)s_2(\k_f^-),\quad \text{for $\k\in S$, $f\in E$, $c_f^-<c_f^+$}
\end{align}
implies $\discmu_1\precsim \discmu_2$.
\end{corollary}
\begin{proof}
Assuming that $s_1$ satisfies \eqref{eq:lattice_FKG} we find using first the assumption
\eqref{eq:stoch_dom_crit1_repeated} and then \eqref{eq:lattice_FKG}
\begin{align}
s_2(\k_{fg}^{++})s_1(\k_{fg}^{--})
\geq \frac{s_1(\k_{fg}^{++})s_2(\k_{fg}^{-+})}{s_1(\k_{fg}^{-+})}s_1(\k_{fg}^{--})\geq s_2(\k_{fg}^{-+})s_1(\k_{fg}^{+-}).
\end{align}
Now Theorem \ref{th:stoch_dom_crit} implies the claim. The proof if $s_2$
satisfies \eqref{eq:lattice_FKG} is similar.
\end{proof}
We can now start to discuss the correlation inequalities for the measures $\discmu^{G,\bs\rho}$. 
Similar results for the simpler case where $\rho$ is as in \eqref{eq:defrhospecial} and homogeneous
have been derived in \cite{phasetransitionclass} where the result was reduced to correlation inequalities for the uniform spanning tree. It turns out that the proofs essentially extend to the present more general setting without changes.
Recall the notation $\kappa^{\pm\pm}_{fg}$ introduced before Theorem \ref{th:stoch_dom}
and also the shorthand $\kappa^{\pm\pm}$. 
 \begin{lemma}\label{le:lattice_2_edges}
For a finite graph $G$ and $\kappa\in S$ as above  
 \begin{align}\label{eq:lattice_2_edges}
 \det \Delta_{\kappa^{++}} \, \det \Delta_{\kappa^{--}}
 \leq\det \Delta_{\kappa^{+-}} \, \det \Delta_{\kappa^{-+}}.
 \end{align}
 \end{lemma}
 \begin{proof}
This is the generalization stated in Remark  4.7 after  Lemma 4.6 in \cite{phasetransitionclass}.
 \end{proof}
The previous lemma directly implies that the measures $\discmu^{G,\bs\rho}$ are strongly positively associated. 
 \begin{corollary}\label{co:Holley1}
The measure $\discmu^{G,\bs\rho}=s_G \bs\rho$ satisfies the FKG lattice condition for any $\kappa_1,\kappa_2\in S$
\begin{align}\label{eq:HolleyCondition}
s_G(\kappa_1 \wedge \kappa_2)s_G(\kappa_1 \vee \kappa_2)\geq s_G(\kappa_1)s_G(\kappa_2)
\end{align}
and the FKG inequality
\begin{align}\label{eq:FKG_measure}
\discmu^{G, \bs \rho}(XY)\geq \discmu^{G, \bs \rho}(X)\, \discmu^{G, \bs \rho}(Y)
\end{align}
for any increasing functions $X,Y:S\to \R$ such that $X$, $Y$, and $XY$ are integrable with respect to $\discmu^{G,\bs\rho}$.
\end{corollary}
\begin{proof}
Note that \eqref{eq:lattice_2_edges} is equivalent to
\begin{align}
s_G(\k^{++})s_G(\k^{--})\geq s_G(\k^{+-})s_G(\k^{-+}).
\end{align}
It follows from Theorem \ref{th:stoch_dom_crit} that the FKG lattice condition  \eqref{eq:HolleyCondition}
holds and therefore by Theorem \ref{th:FKG} also the FKG-inequality \eqref{eq:FKG_measure}.
\end{proof}

We also state correlation inequalities with respect to the size of the graph. More specifically we show statements for subgraphs and contracted graphs. This easily implies 
the existence of infinite volume limits.
 Moreover, we can bound infinite volume states by finite  volume measures in the sense of stochastic domination.
Let $F\subset E$ be a set of edges. We define the contracted graph
$G/F$ by identifying for every edge $f\in F$ the endpoints of $f$.
Similarly for a set $W\subset V$ of vertices we define the contracted graph $G/W$ by identifying all vertices in $W$.
The resulting graphs may have multi-edges.
We also consider connected subgraphs $G'=(V',E')$ of $G$.
Recall the notation $\kappa^{\pm}=\kappa_f^\pm$ for $f\in E$.
We use the notation $s_{G'}(\kappa)=s_{G'}(\kappa{\restriction}_{G'})$ for the density involving the determinant  of the operator $\Delta_\kappa^{G'}$ which denotes the graph Laplacian on $G'$ where we restrict the conductances $\kappa$ 
to $E'$ and similarly we denote by $s_{G/F}(\kappa)$ the density involving the determinant of the graph Laplacian on $G/F$. 
The following lemma relates the determinants of the different graph Laplacians.
\begin{lemma}\label{le:det_incl}
With the notation introduced above we have for $\kappa\in S$
\begin{align}\label{eq:det_inequality_diff_sizes}
\frac{\det \Delta^{G'}_{\kappa^+}}{\det \Delta^{G'}_{\kappa^-}}\geq 
\frac{\det \Delta^G_{\kappa^+}}{\det \Delta^G_{\kappa^-}}\geq
\frac{\det \Delta^{G/F}_{\kappa^+}}{\det \Delta^{G/F}_{\kappa^-}}.
\end{align}
\end{lemma}
\begin{proof}
This is essentially the extension stated in Remark 4.11 of  Lemma 4.10 in \cite{phasetransitionclass}. The extension to the more general setting here is trivial.
\end{proof}
The previous estimates imply correlation inequalities for the measures $\discmu^{G,\bs\rho}$.
To shorten the notation we drop $\bs{\rho}$.
We  introduce the distribution under boundary conditions for a connected subgraph $G'=(V',E')$ of $G$.
For  $\alpha\in \R_+^E$ we define the measure $\discmu^{G,E',\alpha}=s_{G,E',\alpha}(\kappa)\bs\rho$ where the density $s_{G,E',\alpha}:\R_+^{E'}\to \R$ is given by 
\begin{align}\label{eq:def_boundary_condition}
s_{G,E',\alpha}(\kappa)=\frac{1}{Z} \frac{1}{\sqrt{\det \Delta^G_{(\alpha,\kappa)}}}
=
\frac{1}{Z} s_G( (\kappa_{E'}, \alpha_{E'^\crm})).
\end{align}
Here $(\alpha,\kappa)\in \R_+^E$ denotes the conductances given by $\kappa$ on $E'$
and by $\alpha$ on $E\setminus E'$. 
The definition implies that we have the following domain Markov property 
\begin{align}\label{eq:DMP}
\discmu^{G}(\kappa{_{E'}} = \cdot \mid \kappa{_{E\setminus E'}}=\alpha_{E\setminus E'})
=\discmu^{G,E',\alpha}(\cdot) \quad \text{a.s.}
\end{align}
We now state the consequences of Lemma \ref{le:det_incl} on stochastic ordering.
\begin{corollary}\label{co:boundary_conditions}
For a finite graph $G=(V,E)$, a connected subgraph $G'=(V',E')$, an edge subset $F\subset E$, and
configurations $\alpha\in S$  the following 
holds
\begin{align}\label{eq:stoch_dom_incl}
\discmu^{G'}\precsim \discmu^{G,E',\alpha},\qquad  
\discmu^{G,E\setminus F, \alpha}\precsim  \discmu^{G/F}.
\end{align}
\end{corollary}
\begin{proof}
From Lemma \ref{le:det_incl} we obtain for $f\in E'$ and any $\kappa\in \R_+^{E'}$
\begin{align}\label{eq:lattice_incl1}
\frac{s_{G,E',\alpha}(\kappa^+)}{s_{G,E',\alpha}(\kappa^-)}
&\geq 
\frac{s_{G'}(\kappa^+)}{
s_{G'}(\kappa^-)}.
\end{align}
Similarly, Lemma \ref{le:det_incl} implies for $f\in E\setminus F$ and $\kappa \in \R_+^{E\setminus F}$
\begin{align}\label{eq:lattice_incl2}
\frac{s_{G/F}(\kappa^+)}{ 
s_{G/F}(\kappa^-)}
&\geq
\frac{s_{G,E\setminus F,\alpha}(\kappa^+)}{s_{G,E\setminus F,\alpha}(\kappa^-)}.
\end{align}
Then the lattice condition \eqref{eq:lattice_FKG}  and Corollary \ref{co:stoch_dom} imply
the result.
\end{proof}

The next step is to define 
infinite volume measures on $\Z^d$ for the measures $\discmu^{\Lambda, \bs\rho}$.
This requires some additional definitions.
Recall the definition of the $\sigma$-algebras $\discsigma_E$ for $E\subset\Edge(\Z^d)$ and note that there is a similar definition for general graphs which will be used in the following.
An event $A\subset \discsigma$ is called local if it measurable with respect 
to $\discsigma_E$ for some finite set $E$, i.e., $A$ depends only on finitely many edges. 
Similarly we define a local function as a function that is measurable with respect to
$\discsigma_E$ for a finite set $E$.
We say that a sequence of measures $\discmu_n$ on $\R_+^{\Edge(\Z^d)}$ 
converges in the topology of local convergence to a measure $\mu$ if $\mu_n(A)\to \mu(A)$ for all local events $A$. For a background on the choice of topologies in the context of Gibbs measures we refer to \cite{MR2807681}.
 The construction of the infinite volume states proceeds as in \cite{phasetransitionclass} and is similar to 
the construction for the random cluster model.
First, we define infinite volume limits of the finite volume distributions with wired and free boundary conditions.
Let us denote by $\Lambda_n=[-n,n]\cap \Z^d$ the
ball with radius $n$ in the maximum norm around the origin and we denote by $E_n=\Edge(\Lambda_n)$
the edges in $\Lambda_n$.
We introduce the shorthand   $\Lambda^w=\Lambda/\partial\Lambda$ 
for $\Lambda\subset \Z^d$ with wired boundary condition. In particular $\Lambda_n^w$ denotes the box with wired boundary conditions.
We define the finite volume measures on $\Lambda_n$ with free and wired boundary conditions respectively by
\begin{align}\label{eq:def_Gibbs_Lambda_n}
\discmu^{0, \bs\rho}_{n}=\discmu^{\Lambda_n,\bs\rho},\quad \qquad  \discmu^{1, \bs\rho}_{n}=\discmu^{\Lambda_n^w, \bs \rho}.
\end{align}
We again drop $\bs\rho$ from the notation in the following.
From Corollary \ref{co:boundary_conditions} and equation \eqref{eq:DMP}  we conclude that
for any measurable and increasing event $A$ depending only on edges in $E_n$
\begin{align}\label{eq:compare_bc_n_n+1}
\discmu_{n+1}^{0}(A)=\discmu^{\Lambda_{n+1}}(A)=\discmu^{\Lambda_{n+1}}( \discmu^{\Lambda_{n+1},E_n, \kappa}(A))\geq  \discmu^{\Lambda_n}(A)=\discmu_n^{0}(A).
\end{align}
We conclude that for any increasing event $A$ depending only on finitely many edges
the limits $\lim_{n\to \infty} \discmu_{n}^{0}(A)$ and similarly $\lim_{n\to \infty} \discmu_{n}^{1}(A)$ exist.
Using that events that are intersections of intervals of the form $\k_e\in [a, \infty)$
 generate the Borel $\sigma$-algebra we conclude that $\discmu^{0,\bs{\rho}}_n$ converges in the topology of local convergence to a measure $\discmu^{0,\bs{\rho}}$. Moreover, the monotonicity results imply that the same limit is obtained for any sequence $\Lambda\to \Z^d$.
\begin{lemma}\label{le:inf_volume_measures}
The measure $\discmu^{0,\bs\rho}$ and $\discmu^{1,\bs{\rho}}$ satisfy the FKG-inequality and
\begin{align}
\discmu^{0,\bs \rho}\precsim \discmu^{1,\bs\rho}.
\end{align}
\end{lemma}
\begin{proof}
This is a consequence of Corollary \ref{co:Holley1} and Corollary \ref{co:boundary_conditions}
and a limiting argument.
\end{proof}
Now that we have shown the correlation results for arbitrary measures $\bs\rho$ 
we  restrict our attention to the disordered case we are actually interested in.

For a finite graph $G$,  $\xi\in \R^{\Edge(G)}$, and $\rho \in \Mcal(\R_+)$ with support in $[\lambda^{-1}, \lambda]$ and positive mass  we define the distribution 
\begin{align}\label{eq:disc_disorder}
\discmu[\xi](\d\k)=\frac{1}{Z} s_G(\k) \prod_{e\in \Edge(G)} e^{\xi_e\k_e }\, \rho(\d\kappa_e).
\end{align}
For $\Z^d$ we can now define the infinite volume limits $\xi\to \discmu^0[\xi]$ and $\xi\to\discmu^1[\xi]$ as above. We call those measures disordered Gibbs measure
although we refrain from defining a specification that turns them into actual Gibbs measures (cf.\ \cite{phasetransitionclass} for a definition of Gibbs measures for $\rho$ as in \eqref{eq:defrhospecial}).
 We show that they are shift covariant as described in the following lemma. 
\begin{lemma}
The disordered Gibbs measures $\xi\to \discmu^{0}[\xi]$ and $\xi\to \discmu^{1}[\xi]$ are measurable functions that are shift covariant in the sense that  for any continuous local function $f:\R_+^{\Edge(\Z^d)}\to \R$
\begin{align}\label{eq:discmucov1}
\discmu[\xi](f\circ \tau_x)&=\discmu[\tau_x \xi](f)
\end{align}
and, for all $\delta\xi, \xi \in \R^{\Edge(\Z^d)}$  where $\delta\xi$ has finite support and  all bounded, continuous, and
local functions $f:\R_+^{\Edge(\Z^d)}\to \R$
\begin{align}\label{eq:discmucov2}
\discmu[\xi+\delta\xi ](f)=
\frac{\discmu[\xi](f(\k)e^{\sum_{e\in \Edge(\Lambda)} (\delta\xi)_e\kappa_e})}{\discmu[\xi](e^{\sum_{e\in \Edge(\Lambda)} (\delta\xi)_e\kappa_e})}.
\end{align}
\end{lemma}
\begin{proof}
We prove the result for $\discmu^0$ the proof for $\discmu^1$ is similar.
First, we note that the measurability of the functions $\xi\to \discmu^0[\xi]$ is clear since this is the limit of measurable functions (the limiting measures were constructed using the deterministic sequence $\Lambda_n\to \Z^d$).
To prove \eqref{eq:discmucov1} we note that if $f$ only depends on the edges in 
$\Lambda_m$ and $n> m + |x|_\infty$ then \eqref{eq:disc_disorder} implies
\begin{align}
\discmu^{\Lambda_n}[\xi]( f\circ \tau_{x}) 
=\discmu^{\Lambda_n + x }[\tau_x \xi] (f).
\end{align}
The claim follows by sending $n\to \infty$ and the observation that the limit is independent of the chosen sequence $\Lambda\to\Z^d$.
Equation \eqref{eq:discmucov2} follows from a simple calculation.
Assume that $\mathrm{supp}\; \delta\xi\subset \Edge(\Lambda)$ and $f$ only depends on the edges in $\Edge(\Lambda)$. Then we calculate
\begin{align}
\begin{split}
\discmu^\Lambda[\xi+\delta\xi ](f)
&= 
\frac{\int f(\kappa) s_\Lambda(\kappa) e^{\sum_{e\in \Edge(\Lambda)} (\xi_e+(\delta\xi)_e)\kappa_e} \prod_{e\in \Edge(\Lambda)} \rho(\d\k_e)
}{
\int  s_\Lambda(\kappa) e^{\sum_{e\in \Edge(\Lambda)} (\xi_e+(\delta\xi)_e)\kappa_e} \prod_{e\in \Edge(\Lambda)} \rho(\d\k_e)
}
\\
&=
\frac{\int e^{\sum_{e\in \Edge(\Lambda)} (\delta\xi)_e\kappa_e}f(\kappa) s_\Lambda(\kappa) e^{\sum_{e\in \Edge(\Lambda)} \xi_e\kappa_e} \prod_{e\in \Edge(\Lambda)} \rho(\d\k_e)
}{
\int   e^{\sum_{e\in \Edge(\Lambda)} (\delta\xi)_e\kappa_e} s_\Lambda(\kappa) e^{-\sum_{e\in \Edge(\Lambda)} \xi_e\kappa_e} \prod_{e\in \Edge(\Lambda)} \rho(\d\k_e)
}
\\
&=
\frac{\discmu^\Lambda[\xi](f(\kappa)e^{\sum_{e\in \Edge(\Lambda)} (\delta\xi)_e\kappa_e})}
{\discmu^\Lambda[\xi](e^{\sum_{e\in \Edge(\Lambda)} (\delta\xi)_e\kappa_e})}.
\end{split}
\end{align}
The claim now follows as $\Lambda\to \Z^d$.
\end{proof}

\section{Aizenman-Wehr argument for  random conductance model}\label{sec:AW}
In this section we will apply the Aizenman-Wehr argument to the random conductance model to conclude that for almost all realizations $\xi$ the two 
measures $\discmu^1[\xi]$ and $\discmu^0[\xi]$ agree.
Since the state space for the random conductance model is compact the original
proof in \cite{MR1060388} essentially applies to our setting. 
Let us refer also to the slightly simplified and  streamlined presentation can be found  in \cite{MR2252929}.
The following theorem is the main result of this section.
Recall that we consider disorder $(\Omega,\Bcal,\P)$ where $\Omega=\R^{\Edge(\Z^d)}$
such that for $\xi\sim \P$ the  random variables $\xi_e$ are i.i.d.\ and they are  not almost surely constant.
\begin{theorem}\label{th:discrete_disorder}
Let  $(\Omega,\Bcal,\P)$ be disorder as above and $d=2$. Then 
\begin{align}
\discmu^0[\xi]=\discmu^1[\xi]
\end{align}
holds for $\P$-almost all $\xi\in \Omega$.
\end{theorem}

The general strategy to prove this theorem is to show that the effect of the boundary
condition for a volume $\Lambda$ is of order $|\partial\Lambda|$ while
the disorder has by the central limit theorem an effect of order $\sqrt{|\Lambda|}$.
This approach is implemented by controlling the fluctuations of suitable free energies.
 We define the generating function $G_\Lambda$ for a disordered Gibbs measure $\xi\to\discmu[\xi]$ by
 \begin{align}
 G^{\discmu[\xi]}_\Lambda = - \ln \discmu[\xi]\left( e^{-\sum_{e\in \Edge(\Lambda)} \xi_e\kappa_e}\right).
 \end{align}
Suppose that $\discmu$ is a covariant Gibbs measure. Then \eqref{eq:discmucov2} implies 
\begin{align}
G^{\discmu[\xi]}_\Lambda = - \ln\left(\frac{\discmu[\xi-\xi_{\Lambda}](1)}{\discmu[\xi-\xi_\Lambda]\left(e^{\sum_{e\in \Edge(\Lambda)}\xi_e\kappa_e}\right)}\right)
= \ln\discmu[\xi-\xi_\Lambda]\left(e^{ \sum_{e\in \Edge(\Lambda)}\xi_e\kappa_e}\right).
\end{align}
Here $\xi_\Lambda$ denotes the restriction of $\xi$ to $\Edge(\Lambda)$ which will
be extended by zero to $\E(\Z^2)$ if necessary. 
Differentiating with respect to $\xi_f$ for $f\in \Edge(\Lambda)$ we obtain
\begin{align}\label{eq:derivativeG}
\partial_{\xi_f}G^{\discmu[\xi]}_\Lambda = 
\frac{\discmu[\xi-\xi_\Lambda]\left(\k_f e^{\sum_{e\in \Edge(\Lambda)}\xi_e\k_e}\right)}{\discmu[\xi-\xi_\Lambda]\left( e^{\sum_{e\in \Edge(\Lambda)}\xi_e\k_e}\right)}
=
\discmu[\xi](\kappa_f)
\end{align}
where we used again \eqref{eq:discmucov2} in the second step.
Let us define the   $\sigma$-algebra $\Bcal_\Lambda$ that is generated by $(\xi_e)_{e\in \Edge(\Lambda)}$.
Since we want to prove that $\discmu^0[\xi]=\discmu^1[\xi]$ for $\P$-almost all $\xi$ we introduce the function $F_\Lambda:\Omega\to\R$ given by the conditional expectation
\begin{align}
F_\Lambda(\xi) = \mathbb{E}\left( G^{\discmu^1}_\Lambda - G^{\discmu^0}_\Lambda\vert \mathcal{B}_\Lambda\right)(\xi)
\end{align}
Moreover we define the centered version
\begin{align}
\tilde F_\Lambda = F_\Lambda -\E (F_\Lambda).
\end{align}
We  prove a deterministic upper bound and a stochastic lower bound for this quantity. 
Let us start with the upper bound which is model dependent. Here we need to exploit the specific structure of the model which in our case does not directly fit in the standard setting of bounded nearest neighbor interactions.
\begin{lemma}\label{le:upper}
For any $\Lambda\subset \Z^ d$ finite with $|\partial\Lambda|\leq \tfrac12|\Lambda|$ and
for all $\xi\in \R^{\Edge(\Z^d)}$
\begin{align}
|\tilde F_\Lambda(\xi)|\leq 2(2d+1+\ln\lambda)|\partial\Lambda|.
\end{align}
\end{lemma}
\begin{proof}
We are going to show that for all $\xi \in \R^{\Edge(\Z^d)}$
\begin{align}\label{eq:upper_bound1}
\left| {G_\Lambda^{\discmu^1[\xi]}}- {G_\Lambda^{\discmu^0[\xi]}} \right|=\left|\ln\left( \frac{\discmu^0[\xi]\left( e^{-\sum_{e\in \Edge(\Lambda)} \xi_e\kappa_e}\right)}{\discmu^1[\xi]\left( e^{-\sum_{e\in \Edge(\Lambda)} \xi_e\kappa_e}\right)}\right)\right|\leq (2d+1+\ln \lambda)|\partial\Lambda|.
\end{align}
This directly implies the result of the lemma because the conditional expectation is a contraction
and thus $|\E(F_\Lambda)|\leq \sup_{\xi}|F_\Lambda(\xi)|\leq \sup \left| {G_\Lambda^{\discmu^1[\xi]}}-  {G_\Lambda^{\discmu^0[\xi]}} \right|$.
To prove \eqref{eq:upper_bound1} we observe that by equation~\eqref{eq:discmucov2}
we have for a shift covariant Gibbs measure using the notation $\xi_\Lambda^+=\xi_\Lambda \wedge 0$
\begin{align}
\discmu[\xi]\left( e^{-\sum_{e\in \Edge(\Lambda)} \xi_e\kappa_e}\right)
=
\frac{\discmu[\xi-\xi_\Lambda^+](e^{-\sum_{e\in \Edge(\Lambda)} (\xi_e-\xi_e^+) \kappa_e})}
{
\discmu[\xi-\xi_\Lambda^+](e^{\sum_{e\in \Edge(\Lambda)} \xi_e^+ \kappa_e})
}
\end{align}
Thus we obtain
\begin{align}\label{eq:FLambda_decomp}
{G_\Lambda^{\discmu^1[\xi]}}-  {G_\Lambda^{\discmu^0[\xi]}}
= \ln\frac{\discmu^1[\xi-\xi_\Lambda^+](e^{\sum_{e\in \Edge(\Lambda)} \xi_e^+ \kappa_e})}{\discmu^0[\xi-\xi_\Lambda^+](e^{\sum_{e\in \Edge(\Lambda)} \xi_e^+ \kappa_e})}-
\ln \frac{\discmu^1[\xi-\xi_\Lambda^+](e^{-\sum_{e\in \Edge(\Lambda)} (\xi_e-\xi_e^+) \kappa_e})}{\discmu^0[\xi-\xi_\Lambda^+](e^{-\sum_{e\in \Edge(\Lambda)} (\xi_e-\xi_e^+) \kappa_e})}.
\end{align}
We now estimate the first term in the difference.
Note that $e^{\sum_{e\in \Edge(\Lambda)} \xi_e^+ \kappa_e}$ is an increasing local function of $\kappa$. 
Therefore  $\discmu^0[\xi-\xi_\Lambda^+]\precsim \discmu^1[\xi-\xi_\Lambda^+]$
(see Lemma~\ref{le:inf_volume_measures}) implies that
\begin{align}\label{eq:FLambda1}
\frac{\discmu^1[\xi-\xi_\Lambda^+]\left( e^{\sum_{e\in \Edge(\Lambda)} \xi_e^+ \kappa_e}\right)}{\discmu^0[\xi-\xi_\Lambda^+]\left( e^{\sum_{e\in \Edge(\Lambda)} \xi_e^+ \kappa_e}\right)} \geq 1.
\end{align}
We now prove an upper bound. We introduce the notation ${\restriction}_{\Lambda}$ for the marginal
on $\R^{\Edge(\Lambda)}$
of a measure on $\R^{\Edge(\Z^d)}$.
Note that for any $\xi\in \R^{\Edge(\Lambda)}$
we have by Corollary~\ref{co:boundary_conditions} the relations
\begin{align}
\discmu^0_\Lambda[\xi-\xi_\Lambda^+]\precsim \discmu^0[\xi-\xi_\Lambda^+]{\restriction}_{\Lambda}
\quad\text{and} \quad\discmu^1_\Lambda[\xi-\xi_\Lambda^+]\succsim \discmu^1[\xi-\xi_\Lambda^+]{\restriction}_{\Lambda}.
\end{align}
Thus we estimate
\begin{align}
\frac{\discmu^1[\xi-\xi_\Lambda^+]\left( e^{\sum_{e\in \Edge(\Lambda)} \xi_e^+ \kappa_e}\right)}{\discmu^0[\xi-\xi_\Lambda^+]\left( e^{\sum_{e\in \Edge(\Lambda)} \xi_e^+ \kappa_e}\right)}
\leq 
\frac{\discmu^1_\Lambda[\xi-\xi_\Lambda^+]\left( e^{\sum_{e\in \Edge(\Lambda)} \xi_e^+ \kappa_e}\right)}{\discmu^0_\Lambda[\xi-\xi_\Lambda^+]\left( e^{\sum_{e\in \Edge(\Lambda)} \xi_e^+ \kappa_e}\right)}
\end{align}
We now compare the densities of the two measures appearing on the right-hand side of the last equation.
They are given by $s_\Lambda(\kappa)=\sqrt{\det \Delta_\kappa^\Lambda}^{-\tfrac12}$
and $s_{\Lambda^w}(\k)=\sqrt{\det\Delta_\kappa^{\Lambda^w}}^{-\tfrac12}$ respectively.
Lemma~\ref{le:densities_comp} below states that for $\kappa\in [\lambda^{-1},\lambda]^{\Edge(\Lambda)}$
\begin{align}
 \left(\frac{|\Lambda^w|}{|\Lambda|}\right)^{\tfrac12} \left(2^{2d}\lambda\right)^{\tfrac12|\partial\Lambda|}\leq \frac{s_{\Lambda^w}(\k)}{s_\Lambda(\kappa)}\leq \left( 2^{2d}\lambda\right)^{\tfrac12|\partial\Lambda|}.
\end{align}
This implies that 
\begin{align}\label{eq:FLambda2}
\frac{\discmu^1_\Lambda[\xi-\xi_\Lambda^+]\left( e^{\sum_{e\in \Edge(\Lambda)} \xi_e^+ \kappa_e}\right)}{\discmu^0_\Lambda[\xi-\xi_\Lambda^+]\left( e^{\sum_{e\in \Edge(\Lambda)} \xi_e^+ \kappa_e}\right)}\leq 
 \left(\frac{|\Lambda|}{|\Lambda^w|}\right)^{\tfrac12} \left(2^{2d}\lambda\right)^{|\partial\Lambda|}.
\end{align}
The first term  in the decomposition in \eqref{eq:FLambda_decomp}
can be estimated similarly using 
the fact that $e^{-\sum_{e\in \Edge(\Lambda)} (\xi_e-\xi_e^+) \kappa_e}$
is an increasing function of $\kappa$.
This leads in analogy to \eqref{eq:FLambda1} and \eqref{eq:FLambda2} to 
\begin{align}
1\leq \frac{\discmu^1[\xi-\xi_\Lambda^+](e^{-\sum_{e\in \Edge(\Lambda)} (\xi_e-\xi_e^+) \kappa_e})}{\discmu^0[\xi-\xi_\Lambda^+](e^{-\sum_{e\in \Edge(\Lambda)} (\xi_e-\xi_e^+) \kappa_e})}\leq \left(\frac{|\Lambda|}{|\Lambda^w|}\right)^{\tfrac12} \left(2^{2d}\lambda\right)^{|\partial\Lambda|}.
\end{align}
Using the assumption on $\Lambda$ we infer that $|\Lambda| / |\Lambda^w|\leq 2$
and thus we conclude
\begin{align}
\left| {G_\Lambda^{\discmu^1[\xi]}}- {G_\Lambda^{\discmu^0[\xi]}} \right|
\leq \ln \left( \sqrt{2} (2^{2d}\lambda)^{\partial\Lambda}\right)
\leq  \left(1 + 2d+\ln \lambda\right)|\partial\Lambda|.
\end{align}
Thus we have shown \eqref{eq:upper_bound1} and this finishes the proof.
\end{proof}

We now state and prove a lemma that compares the determinants of the free and the wired Laplacian
which was the essential input in the previous lemma.
\begin{lemma}\label{le:densities_comp}
For $\Lambda\subset \Z^d$ and $\kappa\in [\lambda^{-1},\lambda]^{\Edge(\Lambda)}$
the estimate 
\begin{align}
 \frac{|\Lambda^w|}{|\Lambda|} \left(2^{2d}\lambda\right)^{|\partial\Lambda|} \leq \frac{\det \Delta^0_\kappa}{\det\Delta^1_\kappa} \leq  \left( 2^{2d}\lambda\right)^{|\partial\Lambda|}
\end{align}
holds.
\end{lemma}
\begin{proof}
The proof is similar to the estimates in Lemma 5.3 in \cite{phasetransitionclass}.
We denote be $\Delta^0_\kappa$ and $\Delta^1_\kappa$ the graph Laplacian on $\Lambda$ and $\Lambda^w$ respectively. 
We denote the set of spanning trees on $\Lambda$ and $\Lambda^w$   by $T_0$ and $T_1$. 
To compare the determinants of $\Delta_\kappa^0$ and $\Delta_\kappa^1$ we use the Kirchhoff formula for the determinant of a graph Laplacian. 
For a subset $E\subset \Edge(\Lambda)$ we define the weight 
\begin{align}
w(\k,E)=\prod_{e\in E} \k_e.
\end{align}
The Kirchhoff formula in our case reads
\begin{align}
\det \Delta_\k^0= |\Lambda| \sum_{\t\in T_0} w(\k,\t),\quad 
\det \Delta_\k^1= |\Lambda^w| \sum_{t\in T_1} w(\k,\t).
\end{align}
We now define suitable maps from $T_0$ to $T_1$ and vice-versa.
We claim that
there is a map  $\p:T_0 \to T_1$ such that $\p(\t){\restriction_{\mathring\Lambda}}=\t{\restriction_{\mathring\Lambda}}$.
Indeed, removing all edges incident to $\partial \Lambda$ from $\t$ we obtain an acyclic 
subgraph of $\Lambda^w$, hence we can find a tree $\p(\t)$ such that 
\begin{align}
\t{\restriction_{\mathring\Lambda}}\subset \p(\t)\subset \t.
\end{align} 
The observation $|\t\setminus \p(\t)|=|\Lambda|-|\Lambda^w|
=|\partial \Lambda|-1$ implies that 
\begin{align}
w(\kappa,\t)\leq
w(\kappa,\p(\t))\lambda^{|\partial\Lambda|-1}.
\end{align}
Since $\p$ does not change the edges in $\Edge(\mathring\Lambda)$ each tree
$\t\in T_1$ has at most $2^{|\Edge(\Lambda)|-|\Edge(\mathring\Lambda)|}$ preimages.
This can be bounded by the number of outgoing edges of $\partial \Lambda$, i.e.,
$|\Edge(\Lambda)|-|\Edge(\mathring\Lambda)|\leq 2d |\partial \Lambda|$.
Plugging everything together we obtain
\begin{align}\label{eq:compare_Det}
\begin{split}
|\Lambda|^{-1}\det \Delta_\kappa^0=\!\sum_{\t\in T_0}w(\kappa,\t)
\leq \sum_{\t\in T_0} w({\kappa},\p(\t)) \lambda^{|\partial\Lambda|-1}
&\leq 2^{2d|\partial\Lambda|}\lambda^{|\partial\Lambda|}\sum_{\t\in T_1} w({\kappa},\t)
\\
&= 2^{2d|\partial\Lambda|}\lambda^{|\partial\Lambda|}|\Lambda^w|^{-1} \det \Delta_\kappa^1.
\end{split}
\end{align}
Similarly, there is a mapping $\Psi:T_1\to T_0$ such that
$\t\subset \Psi(\t)$ and $\Psi(\t)\setminus \t\subset \Edge(\Lambda)\setminus \Edge(\mathring\Lambda)$. 
Indeed, $\t$ is an acyclic subgraph of $\Lambda$ while 
$\t \cup ( \Edge(\Lambda)\setminus \Edge(\mathring\Lambda))$ is spanning.
Again we can estimate the number of preimages of any tree under $\Psi$ by
$2^{2d|\partial\Lambda|}$ and $w(\k,\t)\leq w(\k,\Psi(\t))\lambda^{|\partial \Lambda|-1}$ since 
$\kappa_e>\lambda^{-1}$ and $|\Psi(\t)\setminus \t|=|\partial\Lambda|-1$.
We get
\begin{align}\begin{split}\label{eq:compare_Det2}
|\Lambda^w|^{-1}\det \Delta_\kappa^1=\sum_{\t\in T_1}w(\kappa,\t)
\leq \sum_{\t\in T_1} w({\kappa},\Psi(\t))\lambda^{|\partial\Lambda|}
&\leq  \lambda^{|\partial\Lambda|} 2^{2d|\partial\Lambda|}\sum_{\t\in T_0} w({\kappa},\t)
\\
&=  \lambda^{|\partial\Lambda|} 2^{2d|\partial\Lambda|}|\Lambda|^{-1}\det \Delta_\kappa^0.
\end{split}
\end{align}
Using \eqref{eq:compare_Det} and \eqref{eq:compare_Det2} we conclude.

\end{proof}

We now consider the lower bound for $\tilde{F}_\Lambda$. This will be a consequence of a general statement in \cite{MR1060388} for fluctuations of functions indexed by $\Lambda\subset \Z^d$ that depend 
on $|\Lambda|$ independent degrees of freedoms.
We call a function $\pi:\Edge(\Z^d)\times\R_+^{\Edge(\Z^d)}\to \R$ shift covariant if
$\pi_{\tau_x e}(\tau_x \xi)=\pi_e(\xi)$.

\begin{proposition}\label{prop:Aizenman}
Let $\Gamma_\Lambda:\R^{\Edge(\Lambda)}\to\R$ be a collection of functions
and $(\xi_e)_{e\in \Edge(\Z^d)}$ i.i.d.\ random variables 
with law $\nu$ such that $\E(e^{t\xi_e})<\infty$ for all $t\in \R$.
 We assume that there exists a shift covariant function
$\pi:\Edge(\Z^d)\times\R_+^{\Edge(\Z^d)}\to \R$  such that for some $M\in \R$ and $K>0$
\begin{align}
\label{eq:cond1}
\frac{\partial \Gamma_\Lambda(\xi_{\Edge(\Lambda)})}{\partial \xi_e}
&=\E(\pi_e\,\vert\, \xi_{\Edge(\Lambda)})\quad \text{for $e\in \Edge(\Lambda)$},
\\
\label{eq:cond2}
\E(\pi_e)&=M,\quad \text{for all $e\in \Edge(\Z^d)$},
\\
\label{eq:cond3}
|\pi_e(\xi)|&\leq 1\quad \text{and}\quad \left| \frac{\partial \pi_e(\xi)}{\partial \xi_e}\right| \leq K \quad \text{uniformly in $\xi$},
\\
\label{eq:cond4}
\E(\Gamma_\Lambda)&=0,
\\
\label{eq:cond5}
\pi_e(\xi)&\geq 0.
\end{align}
Then there is a function $\gamma_\nu:\R\times \R_+\to \R_{+,0}$ 
such that 
\begin{align}
\liminf_{n\to \infty} \E\left( e^{\frac{t \Gamma_{\Lambda_n}}{\sqrt{|\Lambda_n|}}}\right)
\geq e^{\frac{t^2\gamma_\nu(M,K^{-1})}{2}}
\end{align}
and $\gamma_\nu$ has the property that
$\gamma_\nu(\alpha,\beta)>0$ if $\alpha\neq 0$, $\beta>0$, and $\nu$ is not a point mass.
\end{proposition}
\begin{proof}
This result is included in  Proposition 6.1 in \cite{MR1060388}. There it is also shown that $\Gamma_{\Lambda_n} / \sqrt{|\Lambda_n|}$ converges in distribution to a normal variable. 
To pass from their notation to ours, we remark that they consider
disorder indexed by finite subsets of $\Z^d$.
Thus $\xi_e$ with $e=(x,x+e_i)$ in our notation would correspond to 
$\eta_{\{0,e_i\}, x}$ in their notation and similarly for $\pi$. Note that 
they claim that the function $\pi_e$ must be monotone in $\xi_e$ but
an inspection of the proof (cf. the definition of $\gamma_\nu$ in (A.3.2) in \cite{MR1060388}) shows that the correct condition is that $\Gamma_\Lambda(\xi_{\Edge(\Lambda)})$
should be monotone in $\xi$ which by  
\eqref{eq:cond1} is satisfied if $\pi$ is non-negative.
\end{proof}

We now apply this proposition to our setting.
\begin{lemma}\label{le:lower_bound}
The functions $\tilde F_\Lambda$ satisfy
\begin{align}
\liminf_{n\to \infty} \E\left( \exp(t\tilde F_{\Lambda_n} / \sqrt{|\Lambda_n|})\right)\geq 
\exp\left(\frac{t^2b^2}{2}\right)
\end{align}
for some constant $b\geq 0$ where $b$ is positive  if the law of $\xi_e$  is not concentrated on a point mass
and $\E(\discmu^1[\xi](\kappa_e)-\discmu^0[\xi](\kappa_e))>0$.
\end{lemma}
\begin{proof}
We will verify the conditions of Proposition \ref{prop:Aizenman} for
the function $\Gamma_\Lambda = \lambda^{-1} \tilde{F}_\Lambda$.
We set 
\begin{align}
\pi_e(\xi)= \lambda^{-1}\left( \mu^{1}[\xi](\kappa_e)-\mu^0[\xi](\kappa_e)\right).
\end{align}
Shift covariance of $\pi$ follows from the shift covariance of $\mu^1$ and $\mu^0$ stated in equation
\eqref{eq:discmucov1}. 
From \eqref{eq:derivativeG} we conclude that condition \eqref{eq:cond1} holds.
That \eqref{eq:cond2} holds for the same $M$ for all $e$ follows from invariance under lattice symmetries (in fact we only need this for edges parallel to $(0,e_1)$ where it is a consequence of shift invariance).
The first  condition of \eqref{eq:cond3} follows from the assumption that $\kappa_e$ is supported in $[\lambda^{-1},\lambda]$. 
To bound the second term we calculate using
\eqref{eq:discmucov2} twice
\begin{align}
\begin{split}
\partial_{\xi_e}\discmu[\xi](\kappa_e)
=\partial_{\xi_e}\frac{\discmu[\xi-\xi_e]\left(\kappa_e e^{\xi_e\kappa_e}\right)}
{\discmu[\xi-\xi_e]\left(e^{\xi_e\kappa_e}\right)}
&=
\frac{\discmu[\xi-\xi_e]\left(\kappa_e^2 e^{\xi_e\kappa_e}\right)}{\discmu[\xi-\xi_e]\left(e^{\xi_e\kappa_e}\right)}
-\frac{\left(\discmu[\xi-\xi_e]\left(\kappa_e e^{\xi_e\kappa_e}\right)\right)^2}{\left(\discmu[\xi-\xi_e]\left(e^{\xi_e\kappa_e}\right)\right)^2}
\\
&=\discmu[\xi](\kappa_e^2)-\left(\discmu[\xi](\kappa_e)\right)^2.
\end{split}
\end{align}
Using  $\k_e\in [\lambda^{-1},\lambda]$ we have
\begin{align} 
\left|\frac{\partial \sigma_e(\xi)}{\partial \xi_e}\right|
=\lambda^{-1}\left| \mu^{1}[\xi](\kappa_e^2)-\left(\mu^1[\xi](\kappa_e)\right)^2-\mu^0[\xi](\kappa_e)+\left(\mu^0[\xi](\kappa_e)\right)^2\right| \leq \lambda.
\end{align} 
Thus the second part of \eqref{eq:cond3} holds with $K=\lambda$.
Condition \eqref{eq:cond4} is obvious and \eqref{eq:cond5} follows from Lemma~\ref{le:inf_volume_measures}.
Therefore, Proposition \ref{prop:Aizenman} implies, denoting by 
 $\nu$  the law of $\xi_e$,
\begin{align}
\liminf_{n\to \infty} \E\left( e^{\frac{t \Gamma_{\Lambda_n}}{\sqrt{|\Lambda_n|}}}\right)
\geq e^{\frac{t^2\gamma_\nu(M,K^{-1})}{2}}.
\end{align}
where $K=\lambda$ and $M=\lambda^{-1}\E\left( \mu^{1}[\xi](\kappa_e)-\mu^0[\xi](\kappa_e)\right)$. Moreover, $b=\gamma_\nu(M,K^{-1})>0$ is positive if $M\neq 0$ and $\nu$ is not a point mass.
\end{proof}

We can  now prove the almost sure uniqueness of the disordered Gibbs measures.
\begin{proof}[Proof of Theorem \ref{th:discrete_disorder}]
From Lemma~\ref{le:upper} and Lemma~\ref{le:lower_bound}
we conclude that there is a constant $C>0$ such that for all $t\in \R$
\begin{align}
\liminf_{n\to \infty} e^{Ct\frac{|\partial \Lambda_n|}{\sqrt{\Lambda_n|}}}\geq
\liminf_{n\to \infty}\E\left( e^{t \frac{\tilde{F}_{\Lambda_n}}{\sqrt{|\Lambda_n|}}}\right)
\geq e^{\frac{t^2 b^2}{2}}.
\end{align}
Since $|\partial \Lambda_n| / \sqrt{|\Lambda_n|}$ is bounded in dimension $d=2$
this implies, sending $t\to \infty$, that $b=0$.
From Lemma~\ref{le:lower_bound} we conclude that
either the law of $\xi_e$ is  concentrated on a point or 
\begin{align}\label{eq:equal}
\E(\discmu^1(\kappa_e)-\discmu^0(\kappa_e))=0.
\end{align}
Our assumptions rule out the first case. Thus we conclude that \eqref{eq:equal} holds.
Lemma \ref{le:inf_volume_measures} implies  $\discmu^1[\xi](\kappa_e)\geq \discmu^0[\xi](\kappa_e)$
and thus
\begin{align}\label{eq:exp_equal}
\discmu^1[\xi](\kappa_e)=\discmu^0[\xi](\kappa_e)
\end{align}
 for $\P$-almost all $\xi$. A standard argument implies then $\discmu^1[\xi]=\discmu^0[\xi]$ for almost all $\xi$. Indeed, there is a coupling of $\k^1\sim \discmu^1[\xi]$ and $\k^0\sim \discmu^0[\xi]$ such that
$\k^0\leq \k^1$ and from \eqref{eq:exp_equal} we conclude that $\k^0_e=\k^1_e$ almost surely.
\end{proof}

\section{Relation between gradient measures and random conductance model}\label{sec:GMRC}
We now investigate the relation between the random conductance model considered in the previous section and disordered gradient Gibbs measures. This will allow us to  prove our main result. 
The random conductance model and the gradient Gibbs measure will be 
coupled by so-called extended gradient Gibbs measures 
which are measures on $\R_g^{\Edge(\Z^d)}\times \R_+^{\Edge(\Z^d)}$ where
the first marginal is a gradient Gibbs measure and the second marginal will be closely related to the random conductance model. They have been introduced in \cite{MR2322690} and were further studied in \cite{MR2778801, phasetransitionclass}. 
Here we generalize them to the disordered case.
To a  disordered gradient Gibbs state $\mu[\xi]$ we associate an extended gradient Gibbs state $\extmu[\xi]$ which is defined by 
\begin{align}\label{eq:def_ext}
\extmu[\xi](\bs{A}\times \bs{B})
=
\int_{\bs{B}} \bs{\rho}[\xi](\d\kappa) \int_{\bs{A}}\mu[\xi](\d\eta)\,\prod_{e\in E} e^{-\tfrac12 \kappa_e \eta_e^2+V[\xi](\eta_e)}
\end{align}
for $\bs{A} \in \Bcal(\R^{E})$, $\bs{B}\in \Bcal (\R_+^E)$ and $E\subset \Edge(\Z^d)$ finite. This defines a consistent family of measures and can therefore be extended to a measure on $\R_g^{\Edge(\Z^d)}\times \R_+^{\Edge(\Z^d)}$.
We use  conditional distributions  $\extmu(\cdot \vert \Acal)((\eta ,\k))$ for a $\sigma$-algebra $\Acal$.
For the definition and properties of
the conditional distribution we refer to \cite[Section 12]{MR1932358}.
Since we will only consider $\sigma$-algebras that 
respect the product structure, i.e., $\sigma$-algebras generated by $(\kappa_e)_{e\in E}$, $(\eta_e)_{e\in E'}$ for some sets $E,E'\subset \Edge(\Z^d)$ we can view the conditional distributions
as disintegration measures.
We will also use the notation $\mu_{\kappa_E, \eta_{E'}}$ for the conditional distribution
given $\kappa_E$ and $\eta_{E'}$ and we sometimes write $\mu_{\kappa_E, \eta_{E'}}[\xi]$ in the disordered case. 
For a disordered gradient Gibbs measure $\mu$ we denote the extended gradient Gibbs measure by $\extmu$ and the $\k$-marginal of $\extmu$ by $\discmu$. We indicate the
annealed measure by $\mu^a=\E(\mu[\xi])$ and similarly for $\discmu^a$ and $\extmu^a$.
Our first result on extended gradient Gibbs measures characterizes the distribution of $\extmu_\k[\xi]$.
This
is a generalization of Lemma 3.4  in \cite{MR2778801}
to the disordered measures.

\begin{lemma}\label{le:cond_distr}
Let $\xi\to\mu[\xi]$ be a shift covariant disordered gradient Gibbs measure with zero-tilt and ergodic annealed measure
such that 
\begin{align}\label{eq:moment_condition}
\E\left(\mu[\xi](|\eta_e|^{d+\varepsilon})\right)<\infty
\end{align}
for some $\varepsilon>0$ and all $e\in \Edge(\Z^d)$.
Let $\extmu$ be the corresponding 
extended shift covariant gradient Gibbs measure. Then $\extmu_\k[\xi]$ 
viewed as a measure on fields $\p\in \R^{\Z^d}$ with $\p(0)=0$
is almost surely a centered Gaussian field with covariance $(\Delta_\k)^{-1}$. In particular the conditional distribution is independent of $\xi$. 
\end{lemma}
\begin{remark}\label{rmk:ergodic}
Actually, the proof shows that we only need to require that almost all realizations $\eta$ of the annealed measure have zero tilt in the sense that 
\begin{align}\label{eq:sublinear}
\lim_{|x|\to\infty} \frac{1}{|x|} |\p(x)-\p(0)|=0.
\end{align}
\end{remark}
\begin{proof}
We first show that the field is almost surely a mixture of  Gaussians with 
covariance $(\Delta_\k)^{-1}$.

For a gradient field $\eta$ let $\p\in \R^{\Z^d}$ be the unique field such that $\nabla\p=\eta$ and $\p(0)=0$. We define $\lambda^\eta_\Lambda$ as the push-forward
of the measure $\prod_{x\in \overcirc\Lambda^\crm} \delta_{\p(x)}(\d\bar{\p}(x)) \d\bar{\p}_{\overcirc\Lambda}$ along the discrete gradient $\nabla$.
Using that $\extmu[\xi]$ is a disordered Gibbs measure we 
can rewrite 
\begin{align}\label{eq:disordered_Gibbs}
\begin{split}
\int \mu[\xi](\d\eta) F(\eta)
&=
\int \mu[\xi](\d\eta) \int \mu_\Lambda^\eta[\xi]\, F(\eta)
\\
&=
\int \mu[\xi](\d\eta)\, \frac{\int \lambda_\Lambda^\eta(\d\bar\eta)\,  e^{-\sum_{e\in \Edge(\Lambda)} V_e[\xi](\bar\eta_e)}F(\bar\eta)}{Z^\eta_\Lambda[\xi]}.
\end{split}
\end{align}
Where $Z^\eta_\Lambda[\xi]= \int \lambda_\Lambda^\eta(\d\bar\eta)\, e^{-\sum_{e\in \Edge(\Lambda)} V_e[\xi](\bar\eta_e)}$ denotes the normalisation constant.
The last display jointly with \eqref{eq:def_ext}  imply that the conditional distribution of the extended gradient Gibbs measure satisfies 
for almost all $\eta\in \R^{\Edge(\Z^d)}_g$
\begin{align}\label{eq:disordered_decomp}
\begin{split}
\extmu_{\eta_{\Edge(\Lambda)^\crm}}[\xi]
&= \prod_{e\in \Edge(\Lambda)} e^{-\tfrac12 \kappa_e \bar{\eta}_e^2+V[\xi](\bar{\eta})  }\; \bs{\rho}[\xi](\d\kappa)
\discmu[\xi]^{{\eta}}_\Lambda(\d\bar{\eta})
\\
&=\frac{1}{Z_\Lambda^\eta[\xi]} \prod_{e\in \Edge(\Lambda)} e^{-\tfrac12 \kappa_e \bar{\eta}^2} 
 \bs{\rho}[\xi](\d\kappa)
\lambda^{\eta}_\Lambda(\d\bar{\eta}).
\end{split}
\end{align}
We now consider the conditional distribution when
also conditioning on $\k$ and we obtain that almost surely
\begin{align}
(\extmu_\k[\xi])_{\eta_{\Edge(\Lambda)^\crm}}
=\left((\extmu[\xi])_{\eta_{\Edge(\Lambda)^\crm}}\right)_\k
=\frac{  \prod_{e\in \Edge(\Lambda)} e^{-\tfrac12 \kappa_e \bar{\eta}_e^2}\, \lambda^{{\eta}}_\Lambda(\d\bar{\eta}) }{\int\lambda^{{\eta}}_\Lambda(\d\bar{\eta}) \left( \prod_{e\in \Edge(\Lambda)} e^{-\tfrac12 \kappa_e \bar{\eta}_e^2} \right)}.
\end{align}
We conclude that $\extmu_\k[\xi]$
is almost surely a Gibbs measure for the potential
\begin{align}
H_{\Lambda}=\sum_{e\in \Edge(\Lambda)} \tfrac12 \kappa_e\eta_e^2,
\end{align}
i.e., for a Gaussian specification. Theorem 13.24 in \cite{MR2807681} then implies
that $\mu_\k[\xi]$ is almost surely distributed according to a mixture of Gaussians with covariance $(\Delta_\kappa)^{-1}$ such that their means are 
$\Delta_\k$-harmonic.

We now show that in fact the mixture is concentrated on the centered Gaussian distribution.
The annealed measure $\mu^a$ is ergodic with average tilt zero
and it satisfies the moment condition \eqref{eq:moment_condition}. Then we can apply
 a directional ergodic theorem (Theorem 1 in \cite{MR1101082}) 
 which implies that
$\mu^a$ almost surely
\begin{align}
\lim_{|x|\to \infty} \frac{1}{|x|} |\p(x)-\p(0)|=0.
\end{align}
We conclude that almost surely the mean of $\mu_\k[\xi]$ is a sublinear $\Delta_\k$-harmonic function. 

It remains to prove that for almost all $\k$ the constant functions are the only sublinear $\Delta_\k$
harmonic functions.
Theorem 3 in \cite{MR3395463} (see also Example 2.1 there) shows that it is sufficient to
verify that the annealed measure $\discmu^a$ is shift invariant to conclude that $\discmu^a$-almost all $\kappa$ the constant functions are the only sublinear $\Delta_\k$-harmonic functions.

Therefore we are left to show that the distribution of $\k$ under $\discmu^a$ is shift invariant. First we show that for a local event $\bs{B}\in \discsigma$, $x\in \Z^d$ and
almost all $\xi$ 
\begin{align}
\extmu[\xi](\bs{B}) = \extmu[\tau_x\xi] (\tau_x \bs{B}).
\end{align}
Using \eqref{eq:disordered_Gibbs} and shift covariance we calculate for $\kappa\in \R_+^{\Edge(\Z^d)}$ and almost all $\xi\in \Omega$ that
\begin{align}\begin{split}
\int  \mu[\tau_x\xi](\d\eta)
\prod_{e\in \Edge(\Lambda)}& e^{-\tfrac12\k_e {\eta}_e^2+ V[\tau_x\xi]({\eta}_e)}
=
\int \mu[\tau_x\xi](\d{\eta})\,\frac{1}{Z^{{\eta}}_\Lambda[\tau_x\xi]}\int \lambda_{\Lambda}^{{\eta}}(\d\bar{\eta})\;
\prod_{e\in \Edge(\Lambda)} e^{-\tfrac12 \kappa_e \bar{\eta}_e^2}
\\
&=
\int \mu[\xi](\d{\eta})\,\frac{1}{Z^{{\tau_x\eta}}_\Lambda[\tau_x\xi]}\int \lambda_{\Lambda}^{\tau_x{\eta}}(\d\bar{\eta})\;
\prod_{e\in \Edge(\Lambda)} e^{-\tfrac12 \kappa_e \bar{\eta}_e^2}
\\
&=
\int \mu[\xi](\d{\eta})\,\frac{1}{Z^{{\eta}}_{\Lambda-x}[\xi]}\int \lambda_{\Lambda- x}^{{\eta}}(\d\bar{\eta})\;
\prod_{e\in \Edge(\Lambda- x)} e^{-\tfrac12 (\tau_{-x}\kappa)_e \bar{\eta}_e^2}
\\
&=
\int \mu[\xi](\d{\eta}) \,
\prod_{e\in \Edge(\Lambda- x)} e^{-\tfrac12 (\tau_{-x}\kappa)_e {\eta}_e^2+
 V[\xi]({\eta}_e)}
 \end{split}
\end{align}
Recall that $\bs{\rho}_e[\tau_x\xi]=\bs{\rho}_{e-x}[\xi]$.
Using this and the last display, we obtain for $\bs B\in \discsigma_\Lambda$ that
\begin{align}
\begin{split}
\discmu[\tau_x\xi] (\bs B)
&=
\int_{\bs B} \prod_{e\in \Edge(\Lambda)} \bs{\rho}_e[\tau_x \xi](\d\k_e)
\int \mu[\tau_x\xi] (\d\eta)\, e^{-\tfrac12\k_e {\eta}_e^2+ V[\tau_x\xi]({\eta}_e)}
\\
&=
\int_{\bs B} \prod_{e\in \Edge(\Lambda)} \bs{\rho}_e[\tau_x \xi](\d\k_e)
 \int \mu[\xi](\d{\eta}) \;
\prod_{e\in \Edge(\Lambda- x)} e^{-\tfrac12 (\tau_{-x}\kappa)_e {\eta}_e^2+
 V[\xi]({\eta}_e)}
 \\
 &=
 \int \prod_{e\in \Edge(\Lambda-x)} \bs{\rho}_e[ \xi](\d\bar\k_e)\,
\1_{\bs B}(\tau_x\bar\k) \int \mu[\xi](\d{\eta}) \;
\prod_{e\in \Edge(\Lambda- x)} e^{-\tfrac12 \bar\kappa_e{\eta}_e^2+
 V[\xi]({\eta}_e)}
 \\
 &=\discmu[\xi](\tau_{-x}\bs B)
\end{split}
\end{align}
where we used the substitution $\bar\k_e=\k_{e+x}$, i.e., $\bar{\k}=\tau_{-x}\k$ in the second to last step.
Using the shift invariance of the disorder, we conclude that 
the measure $\discmu^a$ is shift invariant, i.e., 
\begin{align}
\discmu^a(\tau_x\bs B)=\discmu^a(\bs B).
\end{align}
This ends the proof.

\end{proof}

Recall that we use the notation $\discmu[\xi]{\restriction}_{\Lambda}$ for the marginal distribution of $\discmu[\xi]$ on $\R^{\Edge(\Lambda)}_+$.
\begin{lemma}\label{le:gradGibbs_stoch_dom}
If $\discmu[\xi]$ is the disordered $\kappa$-marginal of a disordered Gibbs state $\mu[\xi]$ then
\begin{align}
\discmu_\Lambda^0[\xi]\precsim \discmu[\xi]{\restriction}_\Lambda \precsim \discmu_\Lambda^1[\xi].
\end{align}
\end{lemma}
\begin{proof}
The idea of the  proof is the same as in the proof of Proposition 4.18 in \cite{phasetransitionclass} where similar statements were shown.
We first show the second relation which is simpler. The basic observation used in both estimates is that conditioned on the $\eta$ field outside of $\Lambda$ 
we can write down the density of the $\kappa$ field inside $\Lambda$ explicitly
and it is given as the wired conductance model modified by
a corrector term. Let us make this precise using \eqref{eq:disordered_decomp} 
which implies for the conditional distribution of the $\k$-marginal
\begin{align}
\extmu_{\eta_{\Edge(\Lambda)^\crm}}[\xi](\d\kappa )
=Z^{-1}\bs{\rho}_\Lambda[\xi](\d\kappa) \int \lambda_\Lambda^{\eta}(\d\bar\eta)
\prod_{e\in \Edge(\Lambda)}  e^{-\tfrac12 \kappa_e \bar\eta_e^ 2}
\end{align}
where $Z=Z(\eta,\xi)$ is the normalization.
We write $\bar{0}$ for the gradient field with everywhere vanishing gradients.
A simple completion of the square shows that we can decompose the integral in the last display as follows
\begin{align}
\begin{split}\label{eq:rewrite_zero_bc}
\int \lambda_\Lambda^{\eta}(\d\bar\eta)
\prod_{e\in \Edge(\Lambda)}  e^{-\tfrac12 \kappa_e \bar\eta_e^ 2}
&=
e^{-\tfrac12 \sum_{e\in \Edge(\Lambda)}\k_e|\nabla \chi_\Lambda(\k,\eta)(e)|^2}\int \lambda_\Lambda^{\bar{0}}(\d\bar\eta)
\prod_{e\in \Edge(\Lambda)}  e^{-\tfrac12 \kappa_e \bar\eta_e^ 2}
\end{split}
\end{align}
where $\chi_\Lambda(\kappa,\eta):\Lambda\to \R$ denotes the corrector, i.e., the unique
solution of the equation
\begin{align}
\Delta_\k\chi_\Lambda=0 \quad \text{in $\overcirc\Lambda$,}\qquad  \chi_\Lambda=\p\quad \text{on $\partial\Lambda$} 
\end{align} 
where $\p$ is the unique function satisfying $\nabla\p=\eta$ and $\p(0)=0$.
To simplify the notation, we introduce the shorthand
\begin{align}\label{eq:def_of_E}
W_\Lambda(\kappa,\eta)=\sum_{e\in \Edge(\Lambda)}\k_e|\nabla \chi_\Lambda(\k,\eta)(e)|^2
\end{align}
for the energy of the corrector.
Equivalently $\chi_\Lambda$ can be defined as the minimizer of the quadratic form on the right-hand side of the last equation subject to the given boundary values. 
This immediately implies that for fixed $\eta$ the function $W_\Lambda(\kappa,\eta)$ is increasing in $\k$. From \eqref{eq:rewrite_zero_bc} we infer that the distribution 
$\extmu_{\eta_{\Edge(\Lambda)^\crm}}[\xi](\d\k)$ has
density proportional to $e^{-\tfrac12 W_\Lambda(\kappa,\eta)} s_{\Lambda^w}(\kappa)$ and therefore
\begin{align}\label{eq:density_bc}
\extmu_{\eta_{\Edge(\Lambda)^\crm}}[\xi](\d\k)
=\frac{1}{Z} e^{-\tfrac12 W_\Lambda(\kappa,\eta)} \discmu_\Lambda^1[\xi](\d\kappa).
\end{align}
By Corollary \ref{co:Holley1} the FKG inequality holds for the measure $\discmu_\Lambda^1[\xi](\d\kappa)$. Since $W$ is increasing in $\k$ this implies that
$\extmu_{\eta_{\Edge(\Lambda)^\crm}}[\xi](\d\k)
\precsim \discmu^1_\Lambda[\xi]$.
This implies
\begin{align}
\discmu[\xi]{\restriction}_\Lambda \precsim \discmu_\Lambda^1[\xi]
\end{align}
because $\discmu[\xi]{\restriction}_\Lambda$ 
can be expressed as a mixture of such measures.
The second result is shown similarly but the proof is slightly more involved.
Consider $m$ such that $\Lambda\subset \Lambda_m$ and let $E=\Edge(\Lambda)$ and as before $E_m=\Edge(\Lambda_m)$.
For  $\eta\in \R^{\Edge(\Z^d)}_g$ and 
$\bar\k\in \R_+^{\Edge(\Z^d)}$ we consider the measure
$\extmu_{\eta_{E_m^\crm}, \bar\k_{E^\crm}}[\xi](\d\k).
$
From \eqref{eq:density_bc} we conclude that this measure can be expressed as
\begin{align}
\begin{split}\label{eq:dens_subset}
\extmu_{\eta_{E_m^\crm}, \bar\k_{E^\crm}}[\xi](\d\k)
&= \frac{1}{Z}
e^{-\tfrac12 W_{\Lambda_m}((\k_{E},\bar\k_{E^\crm}), \eta)}
s_{\Lambda_m^w}((\k_E,\bar\k_{E^\crm}))\,\bs\rho_{\Lambda_m}[\xi](\d\k)
\\
&=
\frac{1}{Z'} e^{-\tfrac12 W_{\Lambda_m}((\k_E,\bar\k_{E^\crm}), \eta)}
\discmu^{\Lambda_m, E, \bar\kappa}[\xi](\d\k).
\end{split}
\end{align}
By the Giorgi-Nash-Moser theorem the corrector $\chi_{\Lambda_m}$ is Hölder continuous, i.e., there are constants $\alpha=\alpha(\lambda)>0$ and $C=C(\lambda)>0$ such that
for $\k\in [\lambda^{-1},\lambda]^{\Edge(\Z^d)}$ and $e\in \Edge(\Lambda_{m/2})$
\begin{align}\label{eq:hoelder}
|\nabla\chi_{\Lambda_m}(\k,\eta)(e)|\leq \frac{C}{m^\alpha}\,\sup_{x,y\in \partial\Lambda_m} (\p(x)-\p(y)).
\end{align}
The well-known theory for continuous problems was extended to the discrete case 
in \cite[Proposition 6.2]{MR1425544}.
Using that the corrector minimizes the quadratic form in \eqref{eq:def_of_E}
we bound for  $ \tilde{\k},\k\in [\lambda^{-1},\lambda]^{\Edge(\Z^d)}$
such that $\kappa_e=\tilde{\k}_e$ and $e\notin \Edge(\Lambda) $
\begin{align}\label{eq:corrector_diff}
\begin{split}
W_{\Lambda_m}(\k, \eta)
&\leq  \sum_{e\in \Edge(\Lambda_m)}
{\k}_e |\nabla\chi_{\Lambda_m}(\tilde{\k},\eta)(e)|^2
\\ &
\leq 
\sum_{e\in \Edge(\Lambda_m)}
\tilde{\k}_e |\nabla\chi_{\Lambda_m}(\tilde{\k},\eta)(e)|^2
+  \lambda|\Edge(\Lambda)| \sup_{e\in \Edge(\Lambda)}
|\nabla \chi_{\Lambda_m}(\tilde{\k},\eta)(e)|^2.
\end{split}
\end{align}
We now assume that $m$ is sufficiently large such that $\Lambda\subset \Lambda_{m/2}$.
Then we can bound for $  \bar{\k}, \tilde{\k},\k\in [\lambda^{-1},\lambda]^{\Edge(\Z^d)}$ and $\alpha$ and $C$ as above
\begin{align}
|W_{\Lambda_m}((\k_E,\bar\k_{E^\crm}), \eta)
-W_{\Lambda_m}((\k_E,\tilde\k_{E^\crm}), \eta)|
\leq \frac{C\lambda|\Edge(\Lambda)|}{m^{2\alpha}}\left( \sup_{x,y\in \partial\Lambda_m} (\p(x)-\p(y))\right)^2.
\end{align}
 We define the event 
 \begin{align}
 \bs{M}(m)=\left\{\eta\in \R_g^{\Edge(\Z^d)}\,:\,\sup_{x,y\in \partial\Lambda_m} (\p(x)-\p(y))\leq  \ln(m)^3\right\}.
 \end{align}
Lemma~A.1 in \cite{phasetransitionclass} together with Lemma~\ref{le:cond_distr}
imply that for almost all $\xi$
 \begin{align}\label{eq:bound_M}
 \mu[\xi](\bs{M}(m))\geq 1- C \ln(m)^{-1}.
 \end{align} 
 Note that the setting in \cite{phasetransitionclass} 
 is actually restricted to $\k\in \{1,q\}^{\Edge(\Z^d)}$ 
 but the only thing that is really needed is
 that the conductances are bounded below. 
We conclude that for given $\varepsilon$ and   $m\geq m_0(\varepsilon,\lambda)$ sufficiently large, $\eta\in \bs{M}(m)$  uniformly in $\k,\tilde\k,\bar\k\in [\lambda^{-1},\lambda]^{\Edge(\Z^d)}$
\begin{align}
\left|1-e^{W_{\Lambda_m}((\k_E,\bar\k_{E^\crm}), \eta)
-W_{\Lambda_m}((\k_E,\tilde\k_{E^\crm}), \eta)}\right|<\varepsilon.
\end{align}
 Using this estimate in \eqref{eq:dens_subset} we conclude that for any increasing event $\bs{B}\in \discsigma_\Lambda$, $m\geq m_0$, and  $\eta\in \bs{M}(m)$ 
 \begin{align}
 \extmu_{\eta_{E_m^\crm}, \bar\k_{E^\crm}}[\xi](\d\k)
 \geq (1-2\varepsilon)\discmu^{\Lambda_m^w, E, \kappa}[\xi](\bs{B})
 \geq (1-2\varepsilon)\discmu^{\Lambda}(\bs{B})
 \end{align}
 where we used Corollary~\ref{co:boundary_conditions} in the last step.
 We conclude that
 \begin{align}
 \discmu[\xi](\bs{B})\geq (1-2\varepsilon)\discmu^{\Lambda}[\xi](\bs{B})
 - \mu[\xi](\bs{M}(m))
 \end{align}
Sending $m\to \infty$ and $\varepsilon\to 0$ implies together with 
 \eqref{eq:bound_M} that
 \begin{align}
  \discmu[\xi](\bs{B})\geq \discmu^{\Lambda}[\xi](\bs{B}).
 \end{align}
 This ends the proof.

\end{proof}

\begin{proof}[Proof of Theorem \ref{th:main} (Uniqueness)]
Here we prove that there is at most one gradient Gibbs measure satisfying the conditions of the theorem.
Let $\mu$ be a gradient Gibbs measure as in the statement of the theorem.
Then Lemma~\ref{le:gradGibbs_stoch_dom} and a limiting argument show that the $\kappa$-marginal
$\discmu[\xi]$ of the extended gradient Gibbs measure $\extmu[\xi]$ satisfy for almost all $\xi$ the relation $\discmu^0[\xi]\precsim \discmu[\xi] \precsim \discmu^1[\xi]$.
We have shown in Theorem~\ref{th:discrete_disorder} that
for almost all $\xi$ the measures $\discmu^1[\xi]=\discmu^0[\xi]$ agree.
Therefore we can conclude that for almost all $\xi$
the equality $\discmu[\xi]=\discmu^1[\xi]=\discmu^0[\xi]$ holds.
From Lemma~\ref{le:cond_distr} we conclude that the measure $\discmu[\xi]$ determines
$\mu[\xi]$ for almost all $\xi$. This implies uniqueness of the shift covariant measure $\mu$.
\end{proof}
%
It remains to prove the simpler existence part of the theorem. Note that essentially 
existence was sketched in \cite{MR3383338} (see Theorem~3.1 and Remark~3.2 (c) there, their argument is actually restricted to zero-tilt). 
For completeness, we give a self-contained similar proof of the existence result. 
\begin{proof}[Proof of Theorem \ref{th:main} (Existence)]
We consider the following sequence $\extmu_{\Lambda_m}[\xi]$ of measures on 
the space $\R_g^{\Edge(\Z^d)}\times \R_+^{\Edge(\Z^d)}$ given by
\begin{align}
\extmu_{\Lambda_m}[\xi](\d\kappa, \d\eta)
=
\frac{1}{Z} \bs{\rho}[\xi](\d\kappa) \lambda^{\bar{0}}_{\Lambda_m}(\d\eta)\, e^{-\tfrac12 \sum_{e\in \Edge(\Lambda_m)} \k_e\eta_e^2}.
\end{align}
As before, we denote the $\kappa$-marginal of this measure by $\discmu_{\Lambda_m}[\xi]$ and the $\eta$-marginal by $\mu_{\Lambda_m}[\xi]$.
One easily sees that the $\k$-marginal $\discmu_{\Lambda_m}[\xi]$ is given by $Z^{-1}s_{\Lambda^w_m}(\k)\bs{\rho}[\xi]$
while the $\eta$-marginal is given by $\mu^{\bar{0}}_{\Lambda_m}[\xi]$ since 
\begin{align}
\int \bs{\rho}_e[\xi](\d\k)\, e^{-\tfrac12 \k_e\eta_e^2}=e^{-V_e[\xi](\eta_e)}.
\end{align}
Finally we notice that conditioned on $\k$ the gradient field is a centered Gaussian  field with covariance
$(\Delta_\k^{\Lambda^w})^{-1}$.
Taking $\Lambda \to \Z^d$ we observe that the $\kappa$-marginal converges to $\discmu^1[\xi]$.
Thus $\extmu_\Lambda[\xi]$ converges to 
a limiting measure $\extmu[\xi]$ where, given $\k$, the conditional distribution of $\eta$ is centered Gaussian
with covariance $(\Delta_\k)^{-1}$. We denote the law of this Gaussian field by $\mu_\k$.
 Since $\mu_{\Lambda_m}[\xi]$
 is a finite volume Gibbs measure,  a limiting argument shows that the $\eta$-marginal $\mu[\xi]$ of $\extmu[\xi]$ is an infinite volume gradient Gibbs measure.
We now prove that the map $\xi\to \mu[\xi]$ constitutes a disordered gradient Gibbs measure having the desired properties.

The moment condition \eqref{eq:cond_moment} is satisfied by the Brascamp-Lieb inequality, which bounds the moment generating function of the measures $\mu_\k$ uniformly in $\k\in [\lambda^{-1},\lambda]^{\Edge(\Z^d)}$.
Concerning the tilt we notice that already on the level of the quenched measures the tilt vanishes, i.e.,
$\extmu[\xi](\eta_e)=0$ for all $e\in \Edge(\Z^d)$.
Next we show the shift covariance. 
First we rewrite using disintegration that 
 \begin{align}
 \int \extmu[\xi](\d\eta,\d\kappa) \, F(\eta)
 =
 \int \discmu^1[\xi](\d\k) \int \mu_\kappa(\d\eta)\,  F(\eta).
 \end{align}
 We introduce the notation $G(\k)=\int \mu_\k(\d\eta) F(\eta)$.
 The property  $\mu_{tau_x\k} =(\tau_x)_\ast \eta_\k$ implies that
 \begin{align}
 G(\tau_x\k)=\int  \mu_{\tau_x\k}(\d\eta) \, F(\eta)
 =\int \mu_{\k}(\d\eta) F(\tau_x\eta).
 \end{align}
The last two displays and the shift covariance  \eqref{eq:discmucov1}  of $\discmu^1$ imply 
 \begin{align}
 \begin{split}
 \int \extmu[\tau_x\xi](\d\eta,\d\kappa)\, F(\eta)
 =
 \int \discmu^1[\tau_x\xi](\d\k)\, G(\k)
&=
 \int \discmu^1[\xi](\d\k)\, G(\tau_x\k)
 \\ &
 =
 \int \extmu[\xi](\d\eta,\d\kappa) F(\tau_x\eta).
 \end{split}
 \end{align}
It remains to prove that the annealed measure $\mu^a=\E(\mu[\xi])$ is ergodic. 
Here we rely on the uniqueness result proved before. This essentially repeats the proof of Theorem~4.5 in \cite{MR3383338}.
Suppose that the annealed measure $\mu^a$ is not ergodic. Then we can find a shift invariant event $\bs{A}\in \gradsigma$ such that $0<c=\E\left(\mu[\xi](\bs{A})\right)<1$.
Since we know that $\xi\to\mu[\xi]$ is shift covariant, we conclude from the shift invariance of $\bs{A}$ that
$\xi\to \mu[\xi](\bs{A})$ is a shift invariant function. Ergodicity of $\xi$ implies that $\mu[\xi](\bs{A})$ is almost surely constant and equal to $c$. We can
now consider the measures $\xi\to\mu_{\bs{A}}[\xi]$
and $\xi\to\mu_{\bs{A}^\crm}[\xi]$ which are defined by 
\begin{align}
\mu_{\bs{A}}[\xi](\bs{B})=\frac{\mu[\xi](\bs{A}\cap\bs{B})}{c},
\qquad 
\mu_{\bs{A}^\crm}[\xi](\bs{B})=\frac{\mu[\xi](\bs{A}^\crm\cap\bs{B})}{1-c}.
\end{align}   
This defines for almost all $\xi$ a probability measure and
they are distinct disordered gradient Gibbs measures 
by Theorem~7.7(b) and Proposition~14.9 in \cite{MR2807681}.
Shift covariance follows from 
the shift covariance of $\mu[\xi]$ and shift invariance of $\bs{A}$ since
\begin{align}
\mu_{\bs{A}}[\tau_x\xi](F)=\frac{\mu[\tau_x\xi](\mathds{1}_{\bs{A}}F)}{c}
=\frac{\mu[\xi]((\mathds{1}_{\bs{A}}F)\circ \tau_x)}{c}
=\frac{\mu[\xi](\mathds{1}_{{\tau_{-x}\bs{A}}}\,F\circ \tau_x)}{c}
=\mu_{\bs{A}}[\tau_x\xi](F\circ \tau_x).
\end{align}
Note that the definition of $\mu[\xi]$ 
which is a mixture of fields $\mu_\k$ and the Brascamp-Lieb inequality together with a Borel-Cantelli argument imply that 
for $\mu^a$ almost all $\eta$ the bound \eqref{eq:sublinear} holds.
Therefore the uniqueness part of the theorem and Remark~\ref{rmk:ergodic}
imply that $\mu_{\bs{A}}[\xi]=\mu_{\bs{A}^\crm}[\xi]$ for almost all $\xi$. This is a contradiction and we conclude that the annealed measure $\mu^a$ is ergodic.
\end{proof}
\begin{remark}
Actually, it can also directly be shown that the annealed measure is ergodic which yields a proof of existence in arbitrary dimension. 
For this one first shows that the annealed law $\discmu^{1,a}=\E(\discmu^1[\xi])$ is mixing. This can be shown similarly to the case without disorder using the correlation inequalities and the independence of $\xi_e$. Then it remains to show that
$\int \discmu^{1,a}(\d\k)\, \mu_\k$ is ergodic which follows from the decay of correlations of $\mu_\k$. 
\end{remark}

\paragraph{Acknowledgments}
SB has been supported by the 
Deutsche Forschungsgemeinschaft (DFG, German Research Foundation) 
through the Hausdorff Center for Mathematics
(GZ EXC 59 and 2047/1, Projekt-ID 3906$ $85813) and the collaborative research centre
'The mathematics of emerging effects'
(CRC 1060, Projekt-ID   2115$ $04053 ). CC has been supported by the EPSRC grant EP/M027694/1.
CC and SB would like to thank the Isaac Newton Institute for Mathematical Sciences for support and hospitality during the programme \emph{Scaling limits, rough paths, quantum field theory} (supported by EPSRC Grant Number EP/R014604/1) where this project was initiated.

\bibliographystyle{amsplain} 
\bibliography{./New_PHD.bib}

\end{document}